\documentclass[a4]{article}

\usepackage{newtxtext}
\usepackage{amsmath,amsthm,amssymb,amscd,stmaryrd}
\usepackage{tikz-cd}
\usepackage{tikz}
\usetikzlibrary{matrix}
\usepackage{graphicx}
\usepackage{hyperref}
\usepackage{authblk}

\newtheorem{thm}{Theorem}
\newtheorem{lem}[thm]{Lemma}
\newtheorem{prop}[thm]{Proposition}
\newtheorem{cor}[thm]{Corollary}

\newtheorem*{claim}{Claim}

\theoremstyle{definition}

\newtheorem{rem}[thm]{Remark}

\def\Hom{\mathop{\mathrm{Hom}}\nolimits}
\def\End{\mathop{\mathrm{End}}\nolimits}
\def\Aut{\mathop{\mathrm{Aut}}\nolimits}
\def\Ad{\mathop{\mathrm{Ad}}\nolimits}
\def\ad{\mathop{\mathrm{ad}}\nolimits}
\def\id{\mathop{\mathrm{id}}\nolimits}
\def\im{\mathop{\mathrm{im}}\nolimits}
\def\coker{\mathop{\mathrm{coker}}\nolimits}

\def\GL{\mathop{\mathrm{GL}}\nolimits}
\def\SL{\mathop{\mathrm{SL}}\nolimits}
\def\Gr{\mathop{\mathrm{Gr}}\nolimits}
\def\tr{\mathop{\mathrm{tr}}\nolimits}
\def\max{\mathop{\mathrm{max}}\limits}

\newcommand{\mf}[1]{{\mathfrak{#1}}}
\newcommand{\bb}[1]{{\mathbb{#1}}}
\newcommand{\mca}[1]{{\mathcal{#1}}}
\newcommand{\ol}[1]{{\overline{#1}}}

\newcommand{\df}{{d_{\mathcal{F}}}}

\title{Vanishing of cohomology and parameter rigidity of actions of solvable Lie groups}
\author{Hirokazu Maruhashi\thanks{maruhashihirokazu@gmail.com}}
\date{\empty}

\begin{document}
\maketitle

\begin{abstract}
We give a sufficient condition for parameter rigidity of actions of solvable Lie groups, by vanishing of (uncountably many) first cohomologies of the orbit foliations. In some cases, we can prove that vanishing of finitely many cohomologies is sufficient. For this purpose we use a rigidity property of quasiisometry. 

As an application we prove some actions of 2-step solvable Lie groups on mapping tori are parameter rigid. Special cases of these actions are considered in a paper of Matsumoto and Mitsumatsu \cite{MM}. 

We also remark on the relation between transitive locally free actions of solvable Lie groups and lattices in solvable Lie groups, and apply results in rigidity theory of lattices in solvable Lie groups to construct transitive locally free actions with some properties. 
\end{abstract}

\tableofcontents

\section{Introduction}\label{Intro}
Let $\rho_0$ be a $C^\infty$ right action of a connected simply connected solvable Lie group $S$ on a closed $C^\infty$ manifold $M$. We always assume $\rho_0$ is locally free, that is, every isotropy subgroup is discrete in $S$. Then we have a foliation $\mca{F}$ of $M$ by the orbits of $\rho_0$, which is called the orbit foliation of $\rho_0$. We say that $\rho_0$ is {\em parameter rigid} if any $C^\infty$ locally free action $\rho$ of $S$ on $M$ whose orbit foliation coincides with $\mca{F}$ is {\em parameter equivalent} to $\rho_0$. Here parameter equivalence means the following; there are an automorphism $\Phi$ of $S$ and a diffeomorphism $F$ of $M$ such that 
\begin{equation*}
F\left(\rho_0(x,s)\right)=\rho\left(F(x),\Phi(s)\right)
\end{equation*}
for all $x\in M$ and $s\in S$, and $F$ preserves each leaf of $\mca{F}$ and is $C^0$ homotopic to the identity via $C^\infty$ maps preserving each leaf. 

First we review a criterion for parameter rigidity when $S$ is nilpotent. Instead of using $S$, let $N$ denote the acting group. We have the {\em leafwise cohomology} $H^*(\mca{F})$ of the foliation $\mca{F}$, which is defined in a similar way to the usual de Rham cohomology. Using the action $\rho_0$ we can define  a canonical injection $H^*(\mf{n})\hookrightarrow H^*(\mca{F})$. We will always use Fraktur for the corresponding Lie algebras. The author of this article proved in \cite{Ma} and \cite{Ma3} the following: 

\begin{thm}\label{nil}
$\rho_0$ is parameter rigid if and only if $H^1(\mca{F})=H^1(\mf{n})$. 
\end{thm}

This theorem reduces proving parameter rigidity --- which looks a nonlinear problem at first sight --- to a linear one, that is, calculation of first cohomology. 

However this criterion is no longer true when the acting group is a general solvable Lie group. In fact there are a solvable Lie group $S$ and a lattice $\Gamma$ in $S$ for which the natural transitive action $\Gamma\backslash S\curvearrowleft S$ is not parameter rigid while $H^1(\Gamma\backslash S)=H^1(\mf{s})$. We will see it in Section \ref{rel} together with an example of a locally parameter rigid action of a contractible group which is not parameter rigid. These are obtained by looking at the relations between transitive locally free actions and lattices in solvable Lie groups. 

Therefore, the formulation should be changed and we will do it by using twisted leafwise cohomologies. Let us return to the first notations. For the action $M\stackrel{\rho_0}{\curvearrowleft}S$ and a representation $\pi$ of $\mf{s}$ on a finite dimensional real vector space $V$, we will define the leafwise cohomology $H^*\left(\mca{F};\mf{s}\stackrel{\pi}{\curvearrowright}V\right)$ of $\mca{F}$ with coefficient $\pi$ as follows. Let $\omega_0$ denote the canonical 1--form of $\rho_0$, ie for any $x\in M$, $(\omega_0)_x\colon T_x\mca{F}\to\mf{s}$ is the inverse of the derivative at the identity of the map $S\to M$ which sends $g$ to $\rho_0(x,g)$. This is a leafwise $\mf{s}$--valued 1--form satisfying 
\begin{equation*}
d_\mca{F}\omega_0+[\omega_0,\omega_0]=0, 
\end{equation*}
where $d_\mca{F}$ is the leafwise exterior derivative of $\mca{F}$. By composing with $\pi$ we get an $\End(V)$--valued leafwise 1-form $\pi\omega_0$ satisfying 
\begin{equation*}
d_\mca{F}\pi\omega_0+\left[\pi\omega_0,\pi\omega_0\right]=0. 
\end{equation*}
Therefore, the trivial vector bundle $M\times V\to M$ carries the flat leafwise connection whose connection form is $\pi\omega_0$. Here this connection form is relative to any global frame of the bundle which has constant $V$ components. The square of the exterior derivative with respect to this connection on the space of leafwise $V$--valued forms $\Omega^*(\mca{F};V)$ is zero by flatness of the connection. So we obtain the cohomology $H^*\left(\mca{F};\mf{s}\stackrel{\pi}{\curvearrowright}V\right)$. On the other hand the cohomology $H^*\left(\mf{s};\mf{s}\stackrel{\pi}{\curvearrowright}V\right)$ of the Lie algebra $\mf{s}$ with coefficient $\pi$ is defined from the complex $\Hom\left(\bigwedge^*\mf{s},V\right)$. Consider the map 
\begin{equation*}
\Hom\left(\bigwedge^*\mf{s},V\right)\hookrightarrow\Omega^*(\mca{F};V)
\end{equation*}
mapping $\varphi$ to $\omega_0^*\varphi$, where $\omega_0^*$ means pullback by $\omega_0$. This is a cochain map, so that it induces the map between cohomologies. 

\begin{prop}\label{injective}
The induced map $H^*\left(\mf{s};\mf{s}\stackrel{\pi}{\curvearrowright}V\right)\to H^*\left(\mca{F};\mf{s}\stackrel{\pi}{\curvearrowright}V\right)$ is injective. 
\end{prop}

We will prove this in Section \ref{ppin}. Hereafter we shall regard $H^*\left(\mf{s};\mf{s}\stackrel{\pi}{\curvearrowright}V\right)$ as a subspace of $H^*\left(\mca{F};\mf{s}\stackrel{\pi}{\curvearrowright}V\right)$. 

Next we will specify which representations are needed for our sufficient condition. Let $\mf{n}$ be the nilradical of $\mf{s}$, that is, the largest nilpotent ideal of $\mf{s}$. Then $\mf{n}$ contains $[\mf{s},\mf{s}]$ since $[\mf{s},\mf{s}]$ is a nilpotent ideal. Take any subspace $\mf{h}$ satisfying 
\begin{equation*}
[\mf{s},\mf{s}]\subset\mf{h}\subset\mf{n}. 
\end{equation*}
Since $\mf{h}$ is a nilpotent ideal, the descending central series of $\mf{h}$ terminates at some term: $\mf{h}\supset\mf{h}^2\supset\cdots\supset\mf{h}^d\supset 0$. The adjoint representation $\ad$ of $\mf{s}$ has the following invariant filtration: 
\begin{equation*}
\mf{s}\stackrel{\ad}{\curvearrowright}\mf{s}\supset\mf{h}\supset\mf{h}^2\supset\cdots\supset\mf{h}^d\supset 0. 
\end{equation*}
We take the associated graded quotient of this filtration: 
\begin{equation*}
\mf{s}\stackrel{\ad}{\curvearrowright}\Gr_\mf{h}(\mf{s})=\mf{s}/\mf{h}\oplus\bigoplus_{i=1}^d\mf{h}^i/\mf{h}^{i+1}. 
\end{equation*}
Note that $\mf{s}/\mf{h}$ is a direct sum of trivial representations. On this associated graded quotient $\mf{h}$ acts trivially, so that we actually have a representation $\mf{s}/\mf{h}\stackrel{\ad}{\curvearrowright}\Gr_\mf{h}(\mf{s})$. Let $\mca{X}$ denote the set of all surjective homomorphisms $\varphi\colon\mf{s}\to\mf{s}/\mf{h}$ of Lie algebras. An element $\varphi\in\mca{X}$ is just a surjective linear map which vanishes on $[\mf{s},\mf{s}]$ since $\mf{s}/\mf{h}$ is abelian. For any $\varphi\in\mca{X}$ we form a representation $\mf{s}\stackrel{\ad\circ\varphi}{\curvearrowright}\Gr_\mf{h}(\mf{s})$. 

\begin{thm}[Sufficient condition for parameter rigidity]\label{SC}
If 
\begin{equation*}
H^1(\mca{F};\mf{s}\stackrel{\ad\circ\varphi}{\curvearrowright}\Gr_\mf{h}(\mf{s}))=H^1(\mf{s};\mf{s}\stackrel{\ad\circ\varphi}{\curvearrowright}\Gr_\mf{h}(\mf{s}))
\end{equation*}
for some $\mf{h}$ and any $\varphi\in\mca{X}$, then $\rho_0$ is parameter rigid. 
\end{thm}

In general, $[\mf{s},\mf{s}]\neq\mf{n}$ and we choose $\mf{h}$ between them depending on the situations. For instance, we take $\mf{h}=\mf{n}$ in Theorem \ref{volume} while we take $\mf{h}=[\mf{s},\mf{s}]$ in Theorem \ref{ogs}. 

The assumption requires vanishing of the cohomologies {\em for every $\varphi\in\mca{X}$} and this sometimes causes a difficulty when applying the theorem.\footnote{We will see such a situation in Maruhashi \cite{Ma3}. } But it is very likely that for most cases vanishing of the cohomologies for almost all $\varphi\in\mca{X}$ is unnecessary, that is, vanishing for finitely many $\varphi$ is sufficient. To prove such results we have two approaches. First one is the method appearing in Matsumoto and Mitsumatsu \cite[Section 6]{MM} which uses a volume form of the manifold $M$. Here we generalize this method to obtain the next theorem: 

\begin{thm}\label{volume}
Assume the following four conditions. 
\begin{itemize}
\item $\mf{s}$ is nonunimodular. 
\item $\dim\mf{s}/\mf{n}=1$. 
\item $M$ is orientable. 
\item $H^1\left(\mca{F};\mf{s}\stackrel{\ad}{\curvearrowright}\Gr_\mf{n}(\mf{s})\right)=H^1\left(\mf{s};\mf{s}\stackrel{\ad}{\curvearrowright}\Gr_\mf{n}(\mf{s})\right)$. 
\end{itemize}
Then $\rho_0$ is parameter rigid. 
\end{thm}

Recall that $\mf{s}$ is unimodular if and only if $\tr\ad X=0$ for every $X\in\mf{s}$. Note that we have used $\mf{n}$ for $\mf{h}$. In this theorem only vanishing for the natural projection $\mf{s}\twoheadrightarrow\mf{s}/\mf{n}\in\mca{X}$ is required by assuming the first three conditions. 

The next approach is a new one in which we use large scale geometry of solvable Lie groups. This can also be applied to unimodular groups. We can obtain several theorems using this method. Here we appeal to a theorem of Ogasawara \cite{Og}. 
For $i=1,\ldots,k$, let 
\begin{equation*}
A_i=
\begin{pmatrix}
\alpha^{(i)}_1&&\\
&\ddots&\\
&&\alpha^{(i)}_n
\end{pmatrix}
\end{equation*}
be a diagonal matrix with positive diagonal entries and put $A(t)=A_1^{t_1}\cdots A_k^{t_k}$ for $t=(t_1,\ldots,t_k)\in\bb{R}^k$. Here $A_i^{t_i}$ means 
\begin{equation*}
\begin{pmatrix}
\left(\alpha^{(i)}_1\right)^{t_i}&&\\
&\ddots&\\
&&\left(\alpha^{(i)}_n\right)^{t_i}
\end{pmatrix}. 
\end{equation*}
Consider $S_A=\bb{R}^n\rtimes_{A(t)}\bb{R}^k$. Let $W_S$ denote the image of the natural map $\Aut(S)\to\GL\left(S/[S,S]\right)$. We take $\mf{h}=[\mf{s},\mf{s}]$ in the next theorem. 

\begin{thm}\label{ogs}
Assume $S=S_A$ satisfies the next condition: For any $j$ there exists $i$ such that $\alpha^{(i)}_j\neq 1$. If 
\begin{equation*}
H^1\left(\mca{F};\mf{s}\stackrel{\ad\circ\varphi}{\curvearrowright}\Gr_{[\mf{s},\mf{s}]}(\mf{s})\right)=H^1\left(\mf{s};\mf{s}\stackrel{\ad\circ\varphi}{\curvearrowright}\Gr_{[\mf{s},\mf{s}]}(\mf{s})\right)
\end{equation*}
for all $\varphi\in W_S$, then $\rho_0$ is parameter rigid. 
\end{thm}

First $\varphi\in W_S$ is regarded as an element of $\mca{X}$ in the following way. By pulling back through $\exp\colon\mf{s}/[\mf{s},\mf{s}]\simeq S/[S,S]$, $\varphi\colon S/[S,S]\to S/[S,S]$ is regarded as $\varphi\colon\mf{s}/[\mf{s},\mf{s}]\to\mf{s}/[\mf{s},\mf{s}]$. Composing with the natural projection $\mf{s}\twoheadrightarrow\mf{s}/[\mf{s},\mf{s}]$ we get $\varphi\colon\mf{s}\twoheadrightarrow\mf{s}/[\mf{s},\mf{s}]\in\mca{X}$. The second remark is about the condition: For any $j$ there exists $i$ such that $\alpha^{(i)}_j\neq 1$. This is equivalent to $[S,S]=\bb{R}^n$. In this theorem, if we put some genericity condition on $A$, the parameter set $W_S$ becomes finite. 

Here we have used a theorem of Ogasawara, but we can also use other theorems treating rigidity of quasiisometry. We will give other applications in a forthcoming paper. 

Finally we give an application of this method to get parameter rigid actions. We describe somewhat generalized version of the usual construction of suspensions of actions. Let $M_0\stackrel{\rho_0}{\curvearrowleft}H$ be a smooth locally free action of a connected Lie group $H$ on a closed $C^{\infty}$ manifold $M_0$, let $G\stackrel{\Phi}{\curvearrowright}H$ be a smooth action of a connected Lie group $G$ on $H$ by automorphisms and let $\Gamma\stackrel{\lambda}{\curvearrowright}M_0$ be a smooth action of a cocompact lattice $\Gamma$ of $G$ on $M_0$. Assume these three actions satisfy the following compatibility condition: 
\begin{equation*}
\lambda\left(\gamma,\rho_0(x,h)\right)=\rho_0\left(\lambda(\gamma,x),\Phi_\gamma(h)\right)
\end{equation*}
for any $\gamma\in\Gamma, x\in M_0$ and $h\in H$. Let $H\rtimes_\Phi G$ be the semidirect product whose multiplication is defined by 
\begin{equation*}
(h_1,g_1)(h_2,g_2)=\left(h_1\Phi_{g_1}(h_2),g_1g_2\right)
\end{equation*}
for $h_1$, $h_2\in H$ and $g_1$, $g_2\in G$. We define two actions $\Gamma\curvearrowright M_0\times G\curvearrowleft H\rtimes_\Phi G$. The action of $\Gamma$ is defined diagonally: 
\begin{equation*}
\gamma(x,g)=\left(\lambda(\gamma,x),\gamma g\right)
\end{equation*}
for $\gamma\in\Gamma$, $x\in M_0$ and $g\in G$. The action of $H\rtimes_\Phi G$ is defined like the multiplication rule of a semidirect product: 
\begin{equation*}
(x,g)(h,g^\prime)=\left(\rho_0\left(x,\Phi_g(h)\right),gg^\prime\right)
\end{equation*}
for $x\in M_0$, $g$, $g^\prime\in G$ and $h\in H$. Then these two actions commute by the compatibility condition. So we get an action 
\begin{equation*}
\Gamma\backslash(M_0\times G)\stackrel{\rho}{\curvearrowleft}H\rtimes_\Phi G. 
\end{equation*}
This is locally free and the fiber bundle $\Gamma\backslash(M_0\times G)\to\Gamma\backslash G$ with a typical fiber $M_0$ is $\left(H\rtimes_\Phi G\to G\right)$--equivariant. The case in which $H$ is trivial is the usual construction of suspensions. 

We deal with a special case of the above construction. Consider $A\in\GL(n,\bb{Z})$ and an $A$--invariant subspace $V$ of $\bb{R}^n$. Take a one parameter subgroup $\Phi\colon\bb{R}\to\GL(V)$ satisfying $\Phi_1=A|_V$. With respect to the above notation, we let 
\begin{equation*}
M_0=\bb{T}^n=\bb{Z}^n\backslash\bb{R}^n,\quad H=V,\quad G=\bb{R}\quad\text{and}\quad\Gamma=\bb{Z}, 
\end{equation*}
and given three actions are $\bb{Z}^n\backslash\bb{R}^n\curvearrowleft V$ by translations, $\bb{Z}\curvearrowright\bb{T}^n$ by $1\in\bb{Z}$ acting as $A$ and $\bb{R}\stackrel{\Phi}{\curvearrowright}V$, which are compatible in the above sense. The acting group $S=V\rtimes_\Phi\bb{R}$ is solvable and two actions $\bb{Z}\curvearrowright\bb{T}^n\times\bb{R}\curvearrowleft S$ are defined by 
\begin{equation*}
1(x,t)=(Ax,t+1)\quad\text{and}\quad(x,t)(v,s)=\left(x+\Phi_t(v),t+s\right). 
\end{equation*}
The resulting action is $M=\bb{Z}\backslash\left(\bb{T}^n\times\bb{R}\right)\curvearrowleft S$, where $M$ is the mapping torus of $A$. Note that the images $\Phi_\bb{Z}$ and $\Phi_\bb{R}$ of $\bb{Z}$ and $\bb{R}$ by $\Phi$ lie in a real algebraic group $\GL(V)$.  

\begin{thm}\label{mapping torus}
We assume the following four conditions: 
\begin{itemize}
\item $V$ is Diophantine in $\bb{R}^n$. 
\item $\Phi_\bb{Z}$ is Zariski dense in $\Phi_\bb{R}$, meaning $\overline{\Phi_\bb{Z}}=\overline{\Phi_\bb{R}}$. 
\item $1$ is not an eigenvalue of $A|_V$. 
\item $A|_V$ has an eigenvalue whose absolute value is not $1$. 
\end{itemize}
Then $M\curvearrowleft S$ is parameter rigid. 
\end{thm}

The Diophantus condition means that there are a basis $v_1,\ldots,v_p$ of $V$ and positive constants $C$ and $\alpha$ satisfying 
\begin{equation*}
\max_i\left|m\cdot v_i\right|\geq\frac{C}{\|m\|^\alpha}
\end{equation*}
for all $m\in\bb{Z}^n\setminus\{0\}$. 

This theorem\footnote{In a future work, I plan to generalize this and prove it with a simpler,  different calculation. } is a generalization of a theorem of Matsumoto and Mitsumatsu \cite{MM}. They deal with the case when $A$ is hyperbolic with its characteristic polynomial irreducible over $\bb{Q}$ and without eigenvalues on the interval $(-1,0)$, $V$ is the intersection with $\bb{R}^n$ of the direct sum of all eigenspaces of $A\colon\bb{C}^n\to\bb{C}^n$ with eigenvalues of modulus less than $1$, and $\Phi$ is the most naturally defined one. In this case $A$ is diagonalizable since its characteristic polynomial has no multiple roots. So they consider the suspension Anosov flow of the Anosov diffeomorphism $A$ on $\bb{T}^n$ and its weak stable foliation which will be the orbit foliation of the action of $V\rtimes_\Phi\bb{R}$. In this setting the group $V\rtimes_\Phi\bb{R}$ is always non unimodular and Theorem \ref{volume} can be applied. 

In our situation $A$ may have nontrivial Jordan blocks, may have the reducible characteristic polynomial such as 
$A=
\begin{pmatrix}
A_1&\\
&A_2
\end{pmatrix}
$
for some $A_1\in\GL(k,\bb{Z})$, $A_2\in\GL(n-k,\bb{Z})$ or $V$ can be some smaller part of stable directions or a mixture of some stable and unstable directions, in particular $V\rtimes_\Phi\bb{R}$ can be unimodular. The method we use to prove Theorem \ref{mapping torus} is that using large scale geometry instead of that by Matsumoto and Mitsumatsu. 

\section*{Acknowledgements}
This is a part of the Ph.D. Thesis of the author. Most part of the paper was written when the author was a Research Fellow of the Japan Society for the Promotion of Science. I would like to thank the advisor, Masayuki Asaoka, and an anonymous referee, who pointed out the surjectivity in Proposition \ref{corresp}.

\section{General sufficient condition for parameter rigidity}

\subsection{Proof of Proposition \ref{injective}}\label{ppin}
As in Introduction, let $M\stackrel{\rho_0}{\curvearrowleft}S$ be an action and $\mf{s}\stackrel{\pi}{\curvearrowright}V$ be a representation. Let $S\stackrel{\Pi}{\curvearrowright}V$ denote the representation whose differentiation is $\pi$. We define an action $M\times V\curvearrowleft S$ by 
\begin{equation*}
(x,v)s=\left(\rho_0(x,s),\Pi\left(s^{-1}\right)v\right)
\end{equation*}
for $(x,v)\in M\times V$ and $s\in S$ and this turns $M\times V\to M$ into an $S$--equivariant vector bundle. We equip $V$ with a norm which comes from an inner product. Then the space $\Gamma_{\mathrm cont}(V)$ of all continuous sections of the trivial vector bundle $M\times V\to M$ is a Banach space with the supremum norm, and on it, we have a natural representation $S\curvearrowright\Gamma_{\mathrm cont}(V)$ defined by 
\begin{equation*}
(s\xi)(x)=\Pi(s)\xi\left(\rho_0(x,s)\right)
\end{equation*}
for $s\in S$, $\xi\in\Gamma_{\mathrm cont}(V)$ and $x\in M$. We regard $V$ as a subspace of $\Gamma_{\mathrm cont}(V)$ consisting of constant sections. 

\begin{lem}
There is an $S$--equivariant continuous linear map $\mu\colon\Gamma_{\mathrm cont}(V)\to V$ which is the identity on $V$. 
\end{lem}

\begin{proof}
Since $S$ is amenable and $M$ is compact, there exists a $\rho_0$--invariant Borel probability measure $\mu$ on $M$. We define $\mu\colon\Gamma_{\mathrm cont}(V)\to V$ by $\xi\mapsto\int_M\xi d\mu$. Then it is easy to show $\left\|\int_M\xi d\mu\right\|\leq\dim V\left\|\xi\right\|_\infty$ and $\mu(s\xi)=\Pi(s)\mu(\xi)$ for all $\xi\in\Gamma_{\mathrm cont}(V)$ and $s\in S$. 
\end{proof}

Using the map $\mu$, we define a map $r\colon\Omega^*(\mca{F};V)\to\Hom\left(\bigwedge^*\mf{s},V\right)$ which, on the $p$-th degree, takes $\eta$ to $r(\eta)$ defined by 
\begin{equation*}
r(\eta)(X_1,\ldots,X_p)=\mu\left(\eta(X_1,\ldots,X_p)\right)
\end{equation*}
for $X_1,\ldots,X_p\in\mf{s}$. Here $X_1,\ldots,X_p$ are regarded as vector fields on $M$ using $\rho_0$. Namely $X_i$ is regarded as a vector field $x\mapsto(\omega_0)_x^{-1}(X_i)$ tangent to the foliation $\mca{F}$, where $\omega_0$ is the canonical $1$--form of $\rho_0$. We will always do this identification during the paper. 

Before proving that $r$ is a cochain map, let us look at our connection closer. Denote by $\nabla$ the covariant derivative with respect to the flat leafwise connection of the trivial bundle $M\times V\to M$, defined in Section \ref{Intro}. Since the connection form is $\pi\omega_0$, we have $\nabla\xi=\df\xi+\pi\omega_0\xi$ for a section $\xi\in\Omega^0(\mca{F};V)$. Let $D\colon\Omega^p(\mca{F};V)\to\Omega^{p+1}(\mca{F};V)$ denote the covariant exterior derivative arising from $\nabla$. Then it is easy to check that $D\eta=\df\eta+\pi\omega_0\wedge\eta$ for $\eta\in\Omega^p(\mca{F};V)$. 

Next let us see which sections are parallel, that is, sections $\xi\in\Omega^0(\mca{F};V)$ satisfying $\nabla\xi=0$. Fix a point $x_0\in M$ and a vector $v\in V$. Define $\xi_0$ locally along the leaf passing through $x_0$ by $\xi_0\left(\rho_0(x_0,s)\right)=\Pi\left(s^{-1}\right)v$ for $s\in S$ close to the identity. Then for any $y=\rho_0(x_0,s_0)$ with small $s_0\in S$ and any $Y\in\mf{s}$, we have 
\begin{align*}
\nabla_{\left.\frac{d}{dt}\rho_0\left(y,e^{tY}\right)\right\vert_{t=0}}\xi_0&=\df\xi_0\left(\left.\frac{d}{dt}\rho_0\left(y,e^{tY}\right)\right\vert_{t=0}\right)+\pi(Y)\xi_0(y)\\
&=\left.\frac{d}{dt}\Pi\left(e^{-tY}s_0^{-1}\right)v\right\vert_{t=0}+\pi(Y)\Pi\left(s_0^{-1}\right)v\\
&=0. 
\end{align*}
Therefore $\nabla\xi_0=0$ and this means the directions of orbits of the action $M\times V\curvearrowleft S$ is horizontal for the leafwise connection. So we have 
\begin{align*}
\left(\nabla_X\xi\right)(x)&=\lim_{t\rightarrow0}\frac{\Pi\left(e^{tX}\right)\xi\left(\rho_0\left(x,e^{tX}\right)\right)-\xi(x)}{t}\\
&=\lim_{t\rightarrow0}\frac{\left(e^{tX}\xi\right)(x)-\xi(x)}{t}
\end{align*}
for any $\xi\in\Omega^0(\mca{F};V)$, $X\in\mf{s}$ and $x\in M$. 

\begin{lem}\label{uniformly}
$\frac{e^{tX}\xi-\xi}{t}$ converges uniformly to $\nabla_X\xi$ as $t\rightarrow0$. 
\end{lem}

\begin{proof}
Take a basis $v_1,\ldots,v_l$ of $V$ and write $\left(e^{tX}\xi\right)(x)=\sum_{i=1}^lf_i(t,x)v_i$ for some real valued functions $f_i$. Then $\left(\nabla_X\xi\right)(x)=\sum_{i=1}^lf_i^\prime(0,x)v_i$. The function $f_i(t,x)$ has the Taylor expansion 
\begin{equation*}
f_i(t,x)=f_i(0,x)+tf_i^\prime(0,x)+\frac{t^2}{2}f_i^{\prime\prime}\left(\theta_{i,x,t},x\right), 
\end{equation*}
where $\theta_{i,x,t}$ is a number between $0$ and $t$. Since 
\begin{equation*}
\frac{\left(e^{tX}\xi\right)(x)-\xi(x)}{t}-\left(\nabla_X\xi\right)(x)=\frac{t}{2}\sum_{i=1}^lf_i^{\prime\prime}\left(\theta_{i,x,t},x\right)v_i
\end{equation*}
and $f_i^{\prime\prime}(\theta,x)$ is bounded for $-1\leq\theta\leq 1$ and $x\in M$, we get the conclusion. 
\end{proof}

Recall that we have a cochain map $\omega_0^*\colon\Hom\left(\bigwedge^*\mf{s},V\right)\hookrightarrow\Omega^*(\mca{F};V)$. 

\begin{lem}
The map $r$ is a cochain map and $r\circ\omega_0^*$ is the identity. Therefore $\omega_0^*$ induces the injective map between cohomologies. 
\end{lem}

\begin{proof}
Using definitions we verify easily that $r\circ\omega_0^*$ is the identity. 

By Lemma \ref{uniformly}, we have 
\begin{align*}
\mu\left(\nabla_X\xi\right)&=\lim_{t\rightarrow0}\mu\left(\frac{e^{tX}\xi-\xi}{t}\right)\\
&=\lim_{t\rightarrow0}\frac{\Pi\left(e^{tX}\right)\mu(\xi)-\mu(\xi)}{t}\\
&=\pi(X)\mu(\xi) 
\end{align*}
for $X\in\mf{s}$ and $\xi\in\Omega^0(\mca{F};V)$. By this property of $\mu$ and a direct calculation, we see that $r$ is a cochain map. 
\end{proof}

\subsection{Proof of Theorem \ref{SC}}\label{proof of Theorem SC}
Let $M\stackrel{\rho_0}{\curvearrowleft}S$ be an action, $\mca{F}$ be its orbit foliation and $\omega_0$ be the canonical $1$--form of $\rho_0$. The set of all smooth actions $M\curvearrowleft S$ with the orbit foliation $\mca{F}$ is denoted by $A(\mca{F},S)$. We will prove that any $\rho\in A(\mca{F},S)$ is parameter equivalent to $\rho_0$ under the assumption of Theorem \ref{SC}. Let $\omega$ be the canonical $1$--form of $\rho$. To show that $\rho$ is parameter equivalent to $\rho_0$, it is sufficient to prove the existence of some $C^\infty$ map $P\colon M\to S$ and some endomorphism $\Phi\colon S\to S$ satisfying 
\begin{equation*}
\omega=\Ad\left(P^{-1}\right)\Phi_*\omega_0+P^*\Theta, 
\end{equation*}
where $\Theta\in\Omega^1(S;\mf{s})$ is the left Maurer--Cartan form of $S$. See, for instance, Asaoka \cite[Proposition 1.4.4]{A}. In other words, we will show that the $\mf{s}$--valued cocycle $\omega$ is cohomologous to a constant $\mf{s}$--valued cocycle $\Phi_*\omega_0$. We call $\omega$ an $\mf{s}$--valued cocycle because it is the infinitesimal version of a usual $S$--valued cocycle over $\rho_0$. In our situation, usual $K$--valued cocycles over $\rho_0$ for some connected simply connected Lie group $K$ are in one-to-one correspondence with elements $\eta\in\Omega^1(\mca{F};\mf{k})$ satisfying 
\begin{equation*}
\df\eta+[\eta,\eta]=0. 
\end{equation*}
Two $\mf{k}$--valued cocycles $\eta_1$, $\eta_2$ are cohomologous if and only if 
\begin{equation*}
\eta_1=\Ad\left(P^{-1}\right)\eta_2+P^*\Theta_K
\end{equation*}
for some smooth $P\colon M\to K$, where $\Theta_K$ denotes the left Maurer--Cartan form on $K$. Also, $\eta$ is a constant cocycle if and only if $\eta=\Phi_*\omega_0$ for some homomorphism $\Phi\colon S\to K$. See Asaoka \cite[Section 1.4.1]{A} for more details. 

Now recall that we have a filtration 
\begin{equation*}
\mf{s}\supset\mf{h}\supset\mf{h}^2\supset\cdots\supset\mf{h}^d\supset0. 
\end{equation*}
Fix complementary subspaces $V_i$ for $i=0,\ldots,d$ so that 
\begin{equation*}
\mf{s}=V_0\oplus\mf{h}\quad\text{and}\quad\mf{h}^i=V_i\oplus\mf{h}^{i+1}. 
\end{equation*}
The adjoint representations $\mf{s}\curvearrowright\mf{s}/\mf{h}$ and $\mf{s}\curvearrowright\mf{h}^i/\mf{h}^{i+1}$ are canonically identified with representations\footnote{Here we do not assume $V_i$ are invariant under $\ad$. } of $\mf{s}$ on $V_i$, which we call $\pi_i$. The representation $\pi_0$ is a multiple of the trivial representation. For any element $X\in\mf{s}$, let $X^i$ be the $V_i$--component with respect to the decomposition $\mf{s}=\bigoplus_{i=0}^dV_i$, so that we have $X=X^0+X^1+\cdots+X^d$. In this section we use this upper right symbol $i$ as the projection operator onto $V_i$, or to indicate the element belongs to $V_i$. Accordingly, $\omega$ is decomposed as $\omega=\omega^0+\cdots+\omega^d$. 

By looking at the $V_0$--component of the equation $\df\omega+[\omega,\omega]=0$, we get $\df\omega^0=0$. Our assumption $H^1(\mca{F};V_0)=H^1(\mf{s};V_0)$ implies that there are a linear map $\varphi\colon\mf{s}\to V_0$ vanishing on $[\mf{s},\mf{s}]$ and a smooth map $h\colon M\to V_0$ satisfying 
\begin{equation}
\omega^0=\varphi\omega_0+\df h. \label{important}
\end{equation}
We also write $\varphi_\rho$ for $\varphi$. This $\varphi_\rho$ plays an important role (or causes a trouble) in our problem\footnote{By \eqref{important}, $\varphi_\rho=r\left(\omega^0\right)$ for $r$ defined in the previous section. By this formula, we can define $\varphi_\rho$ without any assumption on the cohomology. But in this case the definition might depend on the choice of an $S$--invariant Borel probability measure $\mu$. }. Unexpectedly, to determine $\varphi_\rho$ is not so easy. If we can show $\varphi_\rho$ has some restricted form, we can weaken the assumption of vanishing of cohomologies to a smaller subset of $\mca{X}$. This is what we will do in Section \ref{matmit} and Section \ref{LSG}. Here let us see some properties of $\varphi_\rho$. 

Let $a\colon M\times S\to S$ be the unique smooth map satisfying 
\begin{equation*}
\rho_0(x,s)=\rho\left(x,a(x,s)\right)\quad\text{and}\quad a(x,1)=1. 
\end{equation*}
This is defined since $\rho_0$ and $\rho$ have the same orbit foliation. The map $a$ is a cocycle over $\rho_0$ and is important in our problem. 

\begin{lem}\label{important eq}
For any $x\in M$ and $s\in S$, 
\begin{equation*}
\int_1^s\varphi_\rho\Theta+h\left(\rho_0(x,s)\right)-h(x)=\int_1^{a(x,s)}\Theta^0. 
\end{equation*}
\end{lem}

\begin{proof}
Since $\Theta$ satisfies $d\Theta+[\Theta,\Theta]=0$, we see that $\varphi_\rho\Theta$ and $\Theta^0$ are $V_0$--valued closed $1$--forms on $S$. Fix $x\in M$ and $s\in S$. Then by \eqref{important}, 
\begin{equation*}
\int_x^{\rho_0(x,s)}\left(\varphi_\rho\omega_0+\df h\right)=\int_x^{\rho\left(x,a(x,s)\right)}\omega^0. 
\end{equation*}
These integrals are along a curve contained in a leaf. Take a curve $\gamma(t)$, $0\leq t\leq1$, on $S$ connecting $1$ and $s$. Then the left hand side of the above equation is 
\begin{align*}
&\int_0^1\left\{\varphi_\rho\omega_0\left(\frac{d}{dt}\rho_0\left(x,\gamma(t)\right)\right)+\df h\left(\frac{d}{dt}\rho_0\left(x,\gamma(t)\right)\right)\right\}dt\\
&=\int_0^1\varphi_\rho\Theta\left(\frac{d}{dt}\gamma(t)\right)dt+h\left(\rho_0(x,s)\right)-h(x)
\end{align*}
and to compute the right hand side, take a curve $\gamma_1(t)$ connecting $1$ and $a(x,s)$ and then 
\begin{equation*}
\int_0^1\omega^0\left(\frac{d}{dt}\rho\left(x,\gamma_1(t)\right)\right)dt=\int_0^1\Theta^0\left(\frac{d}{dt}\gamma_1(t)\right)dt. 
\end{equation*}
\end{proof}

\begin{lem}\label{so we can}
The map $\varphi\colon\mf{s}\to V_0$ is surjective. 
\end{lem}

\begin{proof}
Let $q\colon\mf{s}\to\mf{s}/\ker\varphi_\rho$ be the natural projection and let $\bar{\varphi}_\rho\colon\mf{s}/\ker\varphi_\rho\to V_0$ be the induced map from $\varphi_\rho$. The map $S\to V_0$ mapping $s$ to $\int_1^s\Theta^0$ is surjective because $\int_1^{e^X}\Theta^0=X^0$ for every $X\in\mf{s}$. Fix a point $x\in M$. The map $S\to S$ defined by $s\mapsto a(x,s)$ is bijective. See Asaoka \cite[Lemma 1.4.6]{A}. By the above lemma, we have 
\begin{equation*}
\bar{\varphi}_\rho\left(\int_1^sq\Theta\right)=\int_1^{a(x,s)}\Theta^0-h\left(\rho_0(x,s)\right)+h(x). 
\end{equation*}
This and the boundedness of $h$ show $\bar{\varphi}_\rho$ is surjective. 
\end{proof}

Therefore we can regard $\varphi$ as an element of $\mca{X}$. 

\begin{lem}
Set $P=e^h\colon M\to S$. Then 
\begin{equation*}
\Ad(P)\left(\omega-P^*\Theta\right)=\varphi\omega_0+\bar{\omega}^1+\cdots+\bar{\omega}^d
\end{equation*}
for some leafwise $1$--forms $\bar{\omega}^i$ with values in $V_i$. So we have 
\begin{equation}
\omega=\Ad\left(P^{-1}\right)\left(\varphi\omega_0+\bar{\omega}^1+\cdots+\bar{\omega}^d\right)+P^*\Theta.
\end{equation} 
\end{lem}

\begin{proof}
First we show 
\begin{align*}
P^*\Theta&=\sum_{j=0}^\infty(-1)^j\frac{(\ad h)^j}{(j+1)!}\df h\\
&=\df h-\frac{1}{2}(\ad h)\df h+\cdots. 
\end{align*}
This is because for any point $x\in M$ and $X\in T_x\mca{F}$, we have 
\begin{align*}
\left(P^*\Theta\right)(X)&=\Theta\left.\frac{d}{dt}e^{h(x(t))}\right\vert_{t=0}=\left(L_{e^{-h(x)}}\right)_*\left.\frac{d}{dt}e^{h(x(t))}\right\vert_{t=0}\\
&=\left.\frac{d}{dt}e^{-h(x)}e^{h(x)+h(x(t))-h(x)}\right\vert_{t=0}\\
&=\sum_{j=0}^\infty(-1)^j\frac{(\ad h(x))^j}{(j+1)!}Xh, 
\end{align*}
where $x(t)$ is a curve satisfying $\left.\frac{d}{dt}x(t)\right\vert_{t=0}=X$. 

So $\omega-P^*\Theta=\varphi\omega_0+\bar{\bar{\omega}}^1+\cdots+\bar{\bar{\omega}}^d$ for some $\bar{\bar{\omega}}^i$ taking values in $V_i$. Since $\mf{s}/\mf{h}$ is an abelian Lie algebra, $S\stackrel{\Ad}{\curvearrowright}\mf{s}/\mf{h}$ is trivial. Therefore $\Ad(P)\left(\omega-P^*\Theta\right)=\varphi\omega_0+\bar{\omega}^1+\cdots+\bar{\omega}^d$ for some $\bar{\omega}^i$. 
\end{proof}

By this lemma we can replace $\omega$ by a cohomologous cocycle whose $V_0$--component is {\em constant}. We say an element of $\Omega^*(\mca{F};W)$ for some vector space $W$ is {\em constant} if it lies in the image of $\omega_0^*\colon\Hom\left(\bigwedge^*\mf{s},W\right)\hookrightarrow\Omega^*(\mca{F};W)$. Replacing by cohomologous cocycles, we will gradually make components constant, and finally get a constant cocycle. So now we may assume $\omega=\varphi\omega_0+\omega^1+\cdots+\omega^d$ and proceed to the next step.  

\begin{lem}
Assume that 
\begin{equation*}
\omega=\varphi\omega_0+\varphi^1\omega_0+\cdots+\varphi^{k-1}\omega_0+\omega^k+\cdots+\omega^d
\end{equation*}
for some linear maps $\varphi^i\colon\mf{s}\to V_i$, that is, $\omega$ is already constant up to the $V_{k-1}$--component. Then we can choose some smooth $P\colon M\to S$ so that 
\begin{equation*}
\Ad(P)\left(\omega-P^*\Theta\right)=\varphi\omega_0+\varphi^1\omega_0+\cdots+\varphi^{k-1}\omega_0+\varphi^k\omega_0+\bar{\omega}^{k+1}+\cdots+\bar{\omega}^d
\end{equation*}
for some linear map $\varphi^k\colon\mf{s}\to V_k$ and $\bar{\omega}^i$. 
\end{lem}

\begin{proof}
Looking at the $V_k$--component of the equation $\df\omega+[\omega,\omega]=0$, we obtain 
\begin{align*}
0&=\df\omega^k+\left[\varphi\omega_0+\varphi^1\omega_0+\cdots+\varphi^{k-1}\omega_0+\omega^k,\varphi\omega_0+\varphi^1\omega_0+\cdots+\varphi^{k-1}\omega_0+\omega^k\right]^k\\
&=\df\omega^k+\pi_k\varphi\omega_0\wedge\omega^k+\text{constant form}. 
\end{align*}
The $k$ appearing in $[\cdots,\cdots]^k$ in the first line of the above denotes the projection onto $V_k$. Let $D\colon\Omega^p(\mca{F};V_k)\to\Omega^{p+1}(\mca{F};V_k)$ be the covariant exterior derivative arising from the leafwise connection defined by connection form $\pi_k\varphi\omega_0$. We saw $D=\df+\pi_k\varphi\omega_0\wedge$ in the previous section, so that 
\begin{equation}
D\omega^k=\omega_0^*\psi \label{are}
\end{equation}
for some $\psi\in\Hom\left(\bigwedge^2\mf{s},V_k\right)$ by the above computation. Recall $r$ which is defined in the previous section. We set $\theta=r\left(\omega^k\right)\colon\mf{s}\to V_k$. Then 
\begin{equation}
\psi=r\left(\omega_0^*\psi\right)=r\left(D\omega^k\right)=Dr\left(\omega^k\right)=D\theta. \label{kore}
\end{equation}
Here $D$ also denotes the differential of $\Hom\left(\bigwedge^*\mf{s},V_k\right)$. By \eqref{are} and \eqref{kore}, we get 
\begin{equation*}
D\left(\omega^k-\omega_0^*\theta\right)=0. 
\end{equation*}
By Lemma \ref{so we can} and by our assumption, we have $H^1\left(\mca{F};\mf{s}\stackrel{\pi_k\varphi}{\curvearrowright}V_k\right)=H^1\left(\mf{s};\mf{s}\stackrel{\pi_k\varphi}{\curvearrowright}V_k\right)$. Therefore there exist a linear map $\theta^\prime\colon\mf{s}\to V_k$ and a smooth map $h\colon M\to V_k$ satisfying 
\begin{align*}
\omega^k&=\theta\omega_0+\theta^\prime\omega_0+\df h+\pi_k\varphi\omega_0h\\
&=\varphi^k\omega_0+\df h+\pi_k\varphi\omega_0h. 
\end{align*}
Here we set $\varphi^k=\theta+\theta^\prime$. As before we let $P=e^h$ and then 
\begin{equation*}
\omega-P^*\Theta=\varphi\omega_0+\varphi^1\omega_0+\cdots+\varphi^{k-1}\omega_0+\left(\varphi^k\omega_0+\pi_k\varphi\omega_0h\right)+\bar{\bar{\omega}}^{k+1}+\cdots+\bar{\bar{\omega}}^d
\end{equation*}
 for some $\bar{\bar{\omega}}^i$. Finally we compute as follows: 
\begin{align*}
\Ad(P)\left(\omega-P^*\Theta\right)&=e^{\ad h}\left(\varphi\omega_0+\varphi^1\omega_0+\cdots+\varphi^k\omega_0+\pi_k\varphi\omega_0h\right)+\text{higher terms}\\
&=\varphi\omega_0+\varphi^1\omega_0+\cdots+\varphi^k\omega_0+\pi_k\varphi\omega_0h-\pi_k\varphi\omega_0h+\text{higher terms}\\
&=\varphi\omega_0+\varphi^1\omega_0+\cdots+\varphi^k\omega_0+\text{higher terms}. 
\end{align*}
\end{proof}

Applying this lemma repeatedly, we see that given $\omega$ is cohomologous to a cocycle of the form $\omega^\prime=\varphi\omega_0+\varphi^1\omega_0+\cdots+\varphi^d\omega_0$. Set $\Phi_*=\varphi+\varphi^1+\cdots+\varphi^d\colon\mf{s}\to\mf{s}$. Then $\omega^\prime=\Phi_*\omega_0$ is a constant cocycle, because the equation $\df\omega^\prime+\left[\omega^\prime,\omega^\prime\right]=0$ implies that $\Phi_*$ is an endomorphism of the Lie algebra $\mf{s}$. This completes the proof of Theorem \ref{SC}. 

Let $\mca{X}_{\rho_0}$ be the set of all $\varphi\in\mca{X}$ which can be written as $\varphi_\rho$ for some $\rho\in A(\mca{F},S)$ (using the isomorphism $\mf{s}/\mf{h}\simeq V_0$.) What we actually proved in this section is the following: 

\begin{thm}\label{SC2}
Assume $H^1(\mca{F})=H^1(\mf{s})$, so that we can define $\mca{X}_{\rho_0}$ as a subset of $\mca{X}$. If 
\begin{equation*}
H^1\left(\mca{F};\mf{s}\stackrel{\ad\circ\varphi}{\curvearrowright}\Gr_\mf{h}(\mf{s})\right)=H^1\left(\mf{s};\mf{s}\stackrel{\ad\circ\varphi}{\curvearrowright}\Gr_\mf{h}(\mf{s})\right)
\end{equation*}
for all $\varphi\in\mca{X}_{\rho_0}$, then $\rho_0$ is parameter rigid. 
\end{thm}

The next task is trying to prove the set $\mca{X}_{\rho_0}$ is small. 

\begin{rem}
Although we deal with only $\mf{s}$--valued cocycles arising from actions $\rho\in A(\mca{F},S)$, what we do in this section is actually valid for any $\mf{k}$--valued cocycles over $\rho_0$ for any connected simply connected solvable Lie group $K$ and a subspace $\mf{h}$ between $[\mf{k},\mf{k}]$ and the nilradical of $\mf{k}$. Also the acting group $S$ need not be solvable, we need only assume that the action has an invariant Borel probability measure. Solvability is used only for the value group $K$. This might be useful for purposes other than parameter rigidity. 
\end{rem}

\begin{rem}
For actions of semisimple Lie groups or groups with property (T), $\bb{R}$--valued cocycle rigidity implies $K$--valued cocycle rigidity for any connected simply connected solvable Lie groups. This is shown by an obvious argument and valid for actions in broader categories.  
\end{rem}

\begin{rem}
As explained in Asaoka \cite[Section 1.4.4]{A}, $H^1\left(\mca{F};\mf{s}\stackrel{\ad}{\curvearrowright}\mf{s}\right)\left/H^1\left(\mf{s};\mf{s}\stackrel{\ad}{\curvearrowright}\mf{s}\right)\right.$ can be viewed as the formal tangent space at $\rho_0$ in $A(\mca{F},S)/\text{(parameter equivalence)}$. We can show that $H^1\left(\mca{F};\mf{s}\stackrel{\ad}{\curvearrowright}\Gr_\mf{h}(\mf{s})\right)=H^1\left(\mf{s};\mf{s}\stackrel{\ad}{\curvearrowright}\Gr_\mf{h}(\mf{s})\right)$ implies 
\begin{equation*}
H^1\left(\mca{F};\mf{s}\stackrel{\ad}{\curvearrowright}\mf{s}\right)\left/H^1\left(\mf{s};\mf{s}\stackrel{\ad}{\curvearrowright}\mf{s}\right)\right.=0
\end{equation*}
by an argument using spectral sequences. But the converse seems false. 
\end{rem}

\section{Sufficient condition by the method of Matsumoto and Mitsumatsu}\label{matmit}
We prove Theorem \ref{volume} here. Let $M\stackrel{\rho_0}{\curvearrowleft}S$ be an action which we consider and take any $\rho\in A(\mca{F},S)$. According to Theorem \ref{SC2}, what we need to show is $\varphi_\rho$, defined in Section \ref{proof of Theorem SC}, coincides with the natural projection $\mf{s}\twoheadrightarrow\mf{s}/\mf{n}$. In this section we also use the notation $\rho_0^s(x)=\rho_0(x,s)$ and $\rho^s(x)=\rho(x,s)$. As in Section \ref{proof of Theorem SC} we have an $S$--valued cocycle $a\colon M\times S\to S$ over $\rho_0$ satisfying $\rho_0^s(x)=\rho^{a(x,s)}(x)$ for all $x\in M$ and $s\in S$. For any $X\in\mf{s}$, $s\in S$ and $x\in M$, 
\begin{align*}
\left(\rho_0^s\right)_*X_x&=\left.\frac{d}{dt}\rho_0\left(x,e^{tX}s\right)\right\vert_{t=0}\\
&=\left.\frac{d}{dt}\rho_0\left(\rho_0^s(x),e^{t\Ad\left(s^{-1}\right)X}\right)\right\vert_{t=0}\\
&=\left(\Ad\left(s^{-1}\right)X\right)_{\rho_0^s(x)}. 
\end{align*}
Take a basis of $\mf{s}$ and its dual basis of $\mf{s}^*$. The dual basis of $\mf{s}^*$, after being pulled back by the canonical $1$--form $\omega_0$ of $\rho_0$, is regarded as a global frame of the bundle $T^*\mca{F}$. Let $\Omega_0\in\Omega^{\dim S}(\mca{F})$ be the wedge product of the global frame, which is a {\em leafwise volume form}. By the above computation, we see 
\begin{equation*}
\left(\rho_0^s\right)^*\Omega_0=\det\Ad\left(s^{-1}\right)\Omega_0
\end{equation*}
for all $s\in S$. We do the same thing for $\rho$, getting another leafwise volume form $\Omega\in\Omega^{\dim S}(\mca{F})$ which satisfies 
\begin{equation*}
\left(\rho^{s}\right)^*\Omega=\det\Ad\left(s^{-1}\right)\Omega
\end{equation*}
for all $s\in S$. Here we must use the canonical $1$--form $\omega$ of $\rho$ rather than $\omega_0$. Fix a complementary subbundle $E$ to $T\mca{F}$ in $TM$; $TM=T\mca{F}\oplus E$. Since $M$ is orientable, we can choose a nowhere vanishing smooth section $\Omega_\text{tr}$ of $\bigwedge^{\dim E}\left(TM/T\mca{F}\right)^*$. We have natural projections $T\mca{F}\leftarrow TM\to TM/T\mca{F}$ defined by $E$. Let $\ol{\Omega_0}$, $\ol{\Omega}\in\Omega^{\dim S}(M)$ and $\ol{\Omega_\text{tr}}\in\Omega^{\dim E}(M)$ be the pull backs of $\Omega_0$, $\Omega$ and $\Omega_\text{tr}$ by the projections. In this section, bars written over something stand for pulling back something by the projections. Both $\ol{\Omega_0}\wedge\ol{\Omega_\text{tr}}$, $\ol{\Omega}\wedge\ol{\Omega_\text{tr}}\in\Omega^{\dim M}(M)$ are volume forms of $M$. So there is a smooth map $c\colon M\times S\to\bb{R}_{>0}$ satisfying 
\begin{equation*}
(\rho^s)^*\left(\ol{\Omega}\wedge\ol{\Omega_\text{tr}}\right)=c(\cdotp,s)\left(\ol{\Omega}\wedge\ol{\Omega_\text{tr}}\right)
\end{equation*}
for all $s\in S$. It is easy to see that $c$ is a cocycle over $\rho$, that is, 
\begin{equation*}
c(x,ss^\prime)=c(x,s)c\left(\rho(x,s),s^\prime\right)
\end{equation*}
for all $x\in M$ and $s,s^\prime\in S$. Our assumption $H^1(\mca{F})=H^1(\mf{s})$ is equivalent to $\bb{R}$--valued cocycle rigidity of $\rho$. See Maruhashi \cite[Section 2]{Ma} for instance. Thus we can find a homomorphism $\alpha\colon S\to\bb{R}_{>0}$ and a smooth map $P\colon M\to\bb{R}_{>0}$ so that 
\begin{equation*}
c(x,s)=P(x)^{-1}\alpha(s)P\left(\rho(x,s)\right)
\end{equation*}
holds for all $x\in M$ and $s\in S$. Then for any $s\in S$, 
\begin{align*}
(\rho^s)^*\left(P^{-1}\left(\ol{\Omega}\wedge\ol{\Omega_\text{tr}}\right)\right)&=P\left(\rho^s(\cdotp)\right)^{-1}c(\cdotp,s)\left(\ol{\Omega}\wedge\ol{\Omega_\text{tr}}\right)\\
&=\alpha(s)P^{-1}\left(\ol{\Omega}\wedge\ol{\Omega_\text{tr}}\right). 
\end{align*}
By integrating over $M$, we get $\alpha(s)=1$ for all $s\in S$. By replacing $P^{-1}\Omega_\text{tr}$ by $\Omega_\text{tr}$ we may assume that 
\begin{equation}\label{invariant}
(\rho^s)^*\left(\ol{\Omega}\wedge\ol{\Omega_\text{tr}}\right)=\ol{\Omega}\wedge\ol{\Omega_\text{tr}}
\end{equation}
for all $s\in S$. For any $x\in M$ and $s\in S$, two maps 
\begin{equation*}
\left(\rho_0^s\right)_*, \left(\rho^{a(x,s)}\right)_*\colon\left(TM/T\mca{F}\right)_x\to\left(TM/T\mca{F}\right)_{\rho_0^s(x)=\rho^{a(x,s)}(x)}
\end{equation*}
coincide. This is because $\rho_0$ and $\rho$ have the same orbit foliation and it is easy to see this for small $s\in S$. For two small $s$, $s^\prime\in S$, we have 
\begin{equation*}
\begin{tikzcd}[column sep=huge]
\left(TM/T\mca{F}\right)_x\ar[r,"\left(\rho_0^s\right)_*"]\ar[d,equal]\ar[dr,phantom,"\circlearrowright" marking]&\left(TM/T\mca{F}\right)_{\rho_0^s(x)}\ar[r,"\left(\rho_0^{s^\prime}\right)_*"]\ar[d,equal]\ar[dr,phantom,"\circlearrowright" marking]&\left(TM/T\mca{F}\right)_{\rho_0^{ss^\prime}(x)}\ar[d,equal]\\
\left(TM/T\mca{F}\right)_x\ar[r,"\left(\rho^{a(x,s)}\right)_*"']&\left(TM/T\mca{F}\right)_{\rho_0^s(x)}\ar[r,"\left(\rho^{a\left(\rho_0^s(x),s^\prime\right)}\right)_*"']&\left(TM/T\mca{F}\right)_{\rho_0^{ss^\prime}(x)}, \\
\end{tikzcd}
\end{equation*}
so that 
\begin{align*}
\left(\rho_0^{ss^\prime}\right)_*&=\left(\rho_0^{s^\prime}\right)_*\circ\left(\rho_0^s\right)_*=\left(\rho^{a\left(\rho_0^s(x),s^\prime\right)}\right)_*\circ\left(\rho^{a(x,s)}\right)_*\\
&=\left(\rho^{a(x,s)a\left(\rho_0^s(x),s^\prime\right)}\right)_*=\left(\rho^{a\left(x,ss^\prime\right)}\right)_*. 
\end{align*}
Using this we can prove for general $s\in S$. Therefore 
\begin{equation}\label{bimo}
\left(\left(\rho_0^s\right)^*\Omega_\text{tr}\right)_x=\left(\left(\rho^{a(x,s)}\right)^*\Omega_\text{tr}\right)_x
\end{equation}
for all $x\in M$ and $s\in S$. Let $\beta\colon M\times S\to\bb{R}_{>0}$ be the smooth map satisfying $(\rho^s)^*\Omega_\text{tr}=\beta(\cdotp,s)\Omega_\text{tr}$ for all $s\in S$. 

\begin{lem}
We have $\beta(\cdotp,s)=\det\Ad(s)$ for all $s\in S$. Therefore
\begin{equation}\label{ke}
(\rho^s)^*\Omega_\text{tr}=\det\Ad(s)\Omega_\text{tr}. 
\end{equation}
\end{lem}

\begin{proof}
By \eqref{invariant}, 
\begin{equation*}
\ol{\Omega}\wedge\ol{\Omega_\text{tr}}=(\rho^s)^*\left(\ol{\Omega}\wedge\ol{\Omega_\text{tr}}\right)=(\rho^s)^*\ol{\Omega}\wedge(\rho^s)^*\ol{\Omega_\text{tr}}. 
\end{equation*}
Since 
\begin{equation*}
\begin{tikzcd}
T_xM\ar[r,"(\rho^s)_*"]\ar[d]\ar[dr,phantom,"\circlearrowright" marking]&T_{\rho^s(x)}M\ar[d]\\
\left(TM/T\mca{F}\right)_x\ar[r,"(\rho^s)_*"']&\left(TM/T\mca{F}\right)_{\rho^s(x)}, 
\end{tikzcd}
\end{equation*}
we have $(\rho^s)^*\ol{\Omega_\text{tr}}=\ol{(\rho^s)^*\Omega_\text{tr}}$. On the other hand, the diagram 
\begin{equation*}
\begin{tikzcd}
T_xM\ar[r,"(\rho^s)_*"]\ar[d]&T_{\rho^s(x)}M\ar[d]\\
T_x\mca{F}\ar[r,"(\rho^s)_*"']&T_{\rho^s(x)}\mca{F}
\end{tikzcd}
\end{equation*}
does not commute unless $E$ is $\rho$--invariant, so usually $(\rho^s)^*\ol{\Omega}\ne\ol{(\rho^s)^*\Omega}$. But 
\begin{equation*}
(\rho^s)^*\ol{\Omega}\wedge(\rho^s)^*\ol{\Omega_\text{tr}}=\ol{(\rho^s)^*\Omega}\wedge\ol{(\rho^s)^*\Omega_\text{tr}}
\end{equation*}
holds, because $(\rho^s)^*\ol{\Omega}$ and $\ol{(\rho^s)^*\Omega}$ coincide when only vectors in $T\mca{F}$ are substituted, and $\ol{(\rho^s)^*\Omega_\text{tr}}$ vanishes if we substitute a vector in $T\mca{F}$. Therefore 
\begin{align*}
\ol{\Omega}\wedge\ol{\Omega_\text{tr}}&=\ol{(\rho^s)^*\Omega}\wedge\ol{(\rho^s)^*\Omega_\text{tr}}=\ol{\det\Ad\left(s^{-1}\right)\Omega}\wedge\ol{\beta(\cdotp,s)\Omega_\text{tr}}\\
&=\det\Ad\left(s^{-1}\right)\beta(\cdotp,s)\ol{\Omega}\wedge\ol{\Omega_\text{tr}}
\end{align*}
and the claim follows. 
\end{proof}

Combining \eqref{bimo} and \eqref{ke}, we get 
\begin{equation*}
\left(\rho_0^s\right)^*\Omega_\text{tr}=\det\Ad\left(a(\cdotp,s)\right)\Omega_\text{tr}. 
\end{equation*}
Then by the same argument in the proof of the above lemma, 
\begin{align}
\left(\rho_0^s\right)^*\left(\ol{\Omega_0}\wedge\ol{\Omega_\text{tr}}\right)&=\left(\rho_0^s\right)^*\ol{\Omega_0}\wedge\left(\rho_0^s\right)^*\ol{\Omega_\text{tr}}\nonumber\\
&=\ol{\left(\rho_0^s\right)^*\Omega_0}\wedge\ol{\left(\rho_0^s\right)^*\Omega_\text{tr}}\nonumber\\
&=\ol{\det\Ad\left(s^{-1}\right)\Omega_0}\wedge\ol{\det\Ad\left(a(x,s)\right)\Omega_\text{tr}}\nonumber\\
&=\det\Ad\left(s^{-1}\right)\det\Ad\left(a(x,s)\right)\ol{\Omega_0}\wedge\ol{\Omega_\text{tr}}. \label{formula}
\end{align}
This formula will be the key to prove Theorem \ref{volume}. 

\begin{lem}
For any $s\in S$, 
\begin{equation*}
\det\Ad(s)=e^{\tr\ad\int_1^sp\Theta}, 
\end{equation*}
where $p\colon\mf{s}\twoheadrightarrow\mf{s}/\mf{n}$ denotes the natural projection and $\ad$ is also used for $\mf{s}/\mf{n}\stackrel{\ad}{\curvearrowright}\Gr_\mf{n}(\mf{s})$.  
\end{lem}

\begin{proof}
Any $s\in S$ can be written as $s=e^{X_1}\cdots e^{X_k}$ for some $X_1,\ldots,X_k\in\mf{s}$. It is easy to see 
\begin{equation*}
\int_1^sp\Theta=p(X_1)+\cdots+p(X_k). 
\end{equation*}
Since $\tr\ad p(X)=\tr\ad X$ for $X\in\mf{s}$, we have 
\begin{align*}
\det\Ad(s)&=e^{\tr\ad X_1}\cdots e^{\tr\ad X_k}\\
&=e^{\tr\ad\left(p(X_1)+\cdots+p(X_k)\right)}\\
&=e^{\tr\ad\int_1^sp\Theta}. 
\end{align*}
\end{proof}

Since we have $\int_1^s\varphi_\rho\Theta+h\left(\rho_0(x,s)\right)-h(x)=\int_1^{a(x,s)}p\Theta$ as in Lemma \ref{important eq}, the factor appearing in the formula \eqref{formula} will be 
\begin{align*}
\det\Ad\left(s^{-1}\right)\det\Ad\left(a(x,s)\right)&=e^{-\tr\ad\int_1^sp\Theta}e^{\tr\ad\int_1^{a(x,s)}p\Theta}\\
&=\exp\left\{\tr\ad\int_1^s\left(\varphi_\rho-p\right)\Theta+\tr\ad\left(h\left(\rho_0(x,s)\right)-h(x)\right)\right\}. 
\end{align*}
Assume $\varphi_\rho\ne p$ and we will get a contradiction as follows. We have $X\in\mf{s}$ so that $A:=(\varphi_\rho-p)(X)\ne0$. Since $\dim\mf{s}/\mf{n}=1$ and $\mf{s}$ is not unimodular, $\tr\ad A\ne0$. We may assume $\tr\ad A>0$ by replacing $X$. Then for $s=e^{TX}$ with $T>0$ large, the function 
\begin{equation*}
\det\Ad\left(s^{-1}\right)\det\Ad\left(a(x,s)\right)=e^{T\tr\ad A}e^{\tr\ad\left(h\left(\rho_0(x,s)\right)-h(x)\right)}
\end{equation*}
on $M$ must be large because $h$ is bounded. This contradicts the formula \eqref{formula} by integrating over $M$. Therefore $\varphi_\rho=p$ and this completes the proof of Theorem \ref{volume}. 

\begin{rem}
The assumptions $\dim\mf{s}/\mf{n}=1$ and nonunimodularity of $\mf{s}$ are used only in the final part of the proof, so we have proved the following. If $M$ is orientable, $\tr\ad\left(\varphi_\rho-p\right)(X)=0$ for all $X\in\mf{s}$, where $\ad$ denotes $\mf{s}/\mf{n}\stackrel{\ad}{\curvearrowright}\Gr_\mf{n}(\mf{s})$. 
\end{rem}

\section{Sufficient condition given by large scale geometry of solvable Lie groups}\label{LSG}

\subsection{Key Proposition to the method}
Let $X$, $B$ be metric spaces. We say a surjective map $p\colon X\to B$ is a {\em distance respecting projection} if for any two points $b$, $b^\prime\in B$, 
\begin{equation*}
d(b,b^\prime)=d\left(p^{-1}(b),p^{-1}(b^\prime)\right)=d_\mca{H}\left(p^{-1}(b),p^{-1}(b^\prime)\right)
\end{equation*}
holds. Here $d_\mca{H}$ denotes the Hausdorff distance. Let $p\colon X\to B$ and $p^\prime\colon X^\prime\to B^\prime$ be distance respecting projections. For a given diagram 
\begin{equation*}
\begin{tikzcd}
X\ar[r,"f"]\ar[d,"p"']&X^\prime\ar[d,"p^\prime"]\\
B\ar[r,"\varphi"']&B^\prime
\end{tikzcd}
\end{equation*}
we say $f$ is {\em fiber respecting over $\varphi$} if there is $C>0$ such that 
\begin{equation*}
d_\mca{H}\left(f\left(p^{-1}(b)\right),\left(p^\prime\right)^{-1}\left(\varphi(b)\right)\right)<C
\end{equation*}
for all $b\in B$. 

\begin{lem}
Let $G$ be a Lie group and $H$ be a closed subgroup of $G$. 

\noindent
(1) Left invariant Riemannian metrics on $G/H$ are in one-to-one correspondence with inner products on $\mf{g}/\mf{h}$ invariant under $H\stackrel{\Ad}{\curvearrowright}\mf{g}/\mf{h}$ by the canonical identification $\mf{g}/\mf{h}\simeq T_HG/H$. 

\noindent
(2) Assume that there exists an invariant inner product for $H\stackrel{\Ad}{\curvearrowright}\mf{g}/\mf{h}$. Take an inner product of $\mf{g}$ for which the restriction $\mf{h}^\perp\stackrel{\sim}{\to}\mf{g}/\mf{h}$ of the projection $\mf{g}\to\mf{g}/\mf{h}$ is an isometry. Endow $G$ and $G/H$ with left invariant Riemannian metrics defined by these inner products. Let $\pi\colon G\to G/H$ be the projection. Then, for every $g\in G$, the restriction $(L_g)_*\mf{h}^\perp\stackrel{\sim}{\to}T_{gH}G/H$ of $\pi_*\colon T_gG\to T_{gH}G/H$ is an isometry, and the kernel of $\pi_*\colon T_gG\to T_{gH}G/H$ is $(L_g)_*\mf{h}$. 

\noindent
(3) Assume that $G$ is connected. Under the assumption of 2), the map $\pi\colon G\to G/H$ is a distance respecting projection. 
\end{lem}

\begin{proof}
It is easy to check (1) and (2). 

\noindent
(3)\ \ Take arbitrary two points $b$, $b^\prime\in G/H$. 

\begin{claim}
$d\left(\pi^{-1}(b),\pi^{-1}\left(b^\prime\right)\right)\geq d\left(b,b^\prime\right)$. 
\end{claim}

For any $g\in\pi^{-1}(b)$ and $g^\prime\in\pi^{-1}\left(b^\prime\right)$, take a minimal geodesic $\gamma\colon[0,1]\to G$ connecting $g$ and $g^\prime$. Since $\left\|(\pi\circ\gamma)^\prime(t)\right\|=\left\|\pi_*\gamma^\prime(t)\right\|\leq\left\|\gamma^\prime(t)\right\|$ by the property of our metric, we have 
\begin{equation*}
d\left(g,g^\prime\right)=\int_0^1\left\|\gamma^\prime(t)\right\|dt\geq\int_0^1\left\|(\pi\circ\gamma)^\prime(t)\right\|dt\geq d\left(b,b^\prime\right). 
\end{equation*}

\begin{claim}
$d\left(\pi^{-1}(b),\pi^{-1}\left(b^\prime\right)\right)\leq d\left(b,b^\prime\right)$. 
\end{claim}

Take a minimal geodesic $\gamma\colon[0,1]\to G/H$ connecting $b$ and $b^\prime$. Fix a point $g\in\pi^{-1}(b)$. Then there exists a curve $\tilde{\gamma}\colon[0,1]\to G$ starting from $g$ such that $\pi\circ\tilde{\gamma}=\gamma$ and $\tilde{\gamma}^\prime(t)\in(L_{\tilde{\gamma}(t)})_*\mf{h}^\perp$ for all $t\in[0,1]$ by a standard argument. For this curve, we have 
\begin{equation*}
d\left(b,b^\prime\right)=\int_0^1\left\|\gamma^\prime(t)\right\|dt=\int_0^1\left\|\tilde{\gamma}^\prime(t)\right\|dt\geq d\left(\pi^{-1}(b),\pi^{-1}\left(b^\prime\right)\right). 
\end{equation*}

\begin{claim}
$d\left(b,b^\prime\right)=d_\mca{H}\left(\pi^{-1}(b),\pi^{-1}\left(b^\prime\right)\right)$. 
\end{claim}

We have $B\left(\pi^{-1}(b),C\right)=\pi^{-1}\left(B(b,C)\right)$ and $B\left(\pi^{-1}(b^\prime),C\right)=\pi^{-1}\left(B(b^\prime,C)\right)$ by the above discussion, this is obvious. 
\end{proof}

\begin{cor}\label{metric}
Let $G$ be a connected Lie group and $H$ be a connected normal closed subgroup of $G$. Take an inner product of $\mf{g}$. Endow $\mf{g}/\mf{h}$ with the inner product for which the restriction $\mf{h}^\perp\stackrel{\sim}{\to}\mf{g}/\mf{h}$ of the projection $\mf{g}\to\mf{g}/\mf{h}$ is an isometry. Consider left invariant Riemannian metrics on $G$ and $G/H$ corresponding to these inner products. Then the projection $\pi\colon G\to G/H$ is a distance respecting projection. 
\end{cor}

\begin{proof}
This is because $H\stackrel{\Ad}{\curvearrowright}\mf{g}/\mf{h}$ is trivial. 
\end{proof}

Let us return to our previous setting: $M\stackrel{\rho_0}{\curvearrowleft}S$. Take an action $\rho\in A(\mca{F},S)$ and let $a\colon M\times S\to S$ be the cocycle over $\rho_0$ defined by $\rho$, that is, $\rho_0(x,s)=\rho\left(x,a(x,s)\right)$. 

\begin{lem}
For any $x\in M$, the map $S\to S$ taking $s$ to $a(x,s)$ is a biLipschitz diffeomorphism for any left invariant Riemannian metric on $S$. 
\end{lem}

\begin{proof}
See Asaoka \cite[Lemma 1.4.6]{A} for a proof that the map is a diffeomorphism. 

To see that the map is biLipschitz, equip $T\mca{F}$ with the metric $\langle\cdotp,\cdotp\rangle_{\rho_0}$ for which $(\omega_0)_x\colon T_x\mca{F}\stackrel{\sim}{\to}\mf{s}$ is isometric for all $x\in M$. We can easily verify that the metric obtained by pulling back $\langle\cdotp,\cdotp\rangle_{\rho_0}$ through the map $S\to M$ taking $s$ to $\rho_0(x,s)$ is the original Riemannian metric on $S$. Similarly we also have the metric $\langle\cdotp,\cdotp\rangle_\rho$ on $T\mca{F}$ defined by using $\rho$. Since $M$ is compact, there is a constant $C>1$ such that $C^{-1}\|\cdotp\|_{\rho_0}\leq\|\cdotp\|_\rho\leq C\|\cdotp\|_{\rho_0}$. For any $s,s^\prime\in S$, take a minimal geodesic $\gamma\colon [0,1]\to S$ connecting $s$ and $s^\prime$. Then 
\begin{align*}
d\left(s,s^\prime\right)&=\int_0^1\left\|\gamma^\prime(t)\right\|dt=\int_0^1\left\|\frac{d}{dt}\rho_0\left(x,\gamma(t)\right)\right\|_{\rho_0}dt\\
&=\int_0^1\left\|\frac{d}{dt}\rho\left(x,a\left(x,\gamma(t)\right)\right)\right\|_{\rho_0}dt\\
&\geq\frac{1}{C}\int_0^1\left\|\frac{d}{dt}\rho\left(x,a\left(x,\gamma(t)\right)\right)\right\|_\rho dt\\
&=\frac{1}{C}\int_0^1\left\|\frac{d}{dt}a\left(x,\gamma(t)\right)\right\|dt\\
&\geq\frac{1}{C}d\left(a(x,s),a\left(x,s^\prime\right)\right). 
\end{align*}
The inverse is also Lipschitz by the same argument. 
\end{proof}

Recall that $\mf{h}$ is a subspace between $[\mf{s},\mf{s}]$ and $\mf{n}$. Let $K_\rho$ and $H$ be the Lie subgroups corresponding to $\ker\varphi_\rho$ and $\mf{h}$ and $\tilde{\varphi}_\rho\colon S/K_\rho\to S/H$ be the map induced by $\varphi_\rho\colon\mf{s}\to\mf{s}/\mf{h}$. Both $S/K_\rho$ and $S/H$ are vector groups and $\tilde{\varphi}_\rho$ is a linear isomorphism. The following proposition is the key for our method to reduce the set $\mca{X}_{\rho_0}$. 

\begin{prop}\label{key prop}
For any $\rho\in A(\mca{F},S)$ and $x\in M$, consider the next diagram 
\begin{equation*}
\begin{tikzcd}
S\ar[r,"{a(x,\cdotp)}"]\ar[d,"\tilde{q}"']&S\ar[d,"\tilde{p}"]\\
S/K_\rho\ar[r,"\sim","\tilde{\varphi}_\rho"']&S/H, 
\end{tikzcd}
\end{equation*}
where the vertical maps are the natural projections. We give $S$ a left invariant Riemannian metric and $S/K_\rho$, $S/H$ the left invariant Riemannian metrics considered in Corollary \ref{metric}. Then $a(x,\cdotp)$ is a fiber respecting biLipschitz diffeomorphism over $\tilde{\varphi}_\rho$. 
\end{prop}

\begin{proof}
Let $q\colon\mf{s}\to\mf{s}/\ker\varphi_\rho$ and $p\colon\mf{s}\to\mf{s}/\mf{h}$ be the natural projections. Then we have 
\begin{align*}
e^{\int_1^sq\Theta}&=\tilde{q}(s), \\
e^{\int_1^sp\Theta}&=\tilde{p}(s)
\end{align*}
for all $s\in S$. Indeed, writing $s=e^{X_1}\cdots e^{X_k}$ for some $X_1,\ldots,X_k\in\mf{s}$ and then 
\begin{align*}
e^{\int_1^sq\Theta}&=\exp\left(\int_1^{e^{X_1}}q\Theta+\int_{e^{X_1}}^{e^{X_1}e^{X_2}}q\Theta+\cdots+\int_{e^{X_1}\cdots e^{X_{k-1}}}^sq\Theta\right)\\
&=e^{q(X_1)+\cdots+q(X_k)}. 
\end{align*}
Let $\bar{\varphi}_\rho\colon\mf{s}/\ker{\varphi_\rho}\to\mf{s}/\mf{h}$ be the map induced by $\varphi_\rho$. By taking $\exp$ to the equation $\bar{\varphi}_\rho\left(\int_1^sq\Theta\right)+h\left(\rho_0(x,s)\right)-h(x)=\int_1^{a(x,s)}p\Theta$ from Lemma \ref{important eq}, we get 
\begin{equation*}
\tilde{\varphi}_\rho\tilde{q}(s)e^{h\left(\rho_0(x,s)\right)-h(x)}=\tilde{p}\left(a(x,s)\right). 
\end{equation*}
Since $M$ is compact, there is a constant $C>0$ such that 
\begin{equation}\label{gosa}
d\left(\tilde{\varphi}_\rho\tilde{q}(s),\tilde{p}\left(a(x,s)\right)\right)=d\left(1,e^{h\left(\rho_0(x,s)\right)-h(x)}\right)<C
\end{equation}
for any $s\in S$. This means that the diagram commutes up to bounded distance. Let $f$ be the inverse map of $a(x,\cdotp)$. By Lemma \ref{important eq}, 
\begin{equation*}
\int_1^{f(s)}q\Theta+\bar{\varphi}_\rho^{-1}\left(h\left(\rho_0\left(x,f(s)\right)\right)-h(x)\right)=\bar{\varphi}_\rho^{-1}\left(\int_1^sp\Theta\right). 
\end{equation*}
Taking $\exp$ to both sides of the equation, 
\begin{equation*}
\tilde{q}f(s)e^{\bar{\varphi}_\rho^{-1}\left(h\left(\rho_0\left(x,f(s)\right)\right)-h(x)\right)}=\tilde{\varphi}_\rho^{-1}\tilde{p}(s). 
\end{equation*}
As before there is a constant $C^\prime>0$ such that 
\begin{equation}\label{gosa2}
d\left(\tilde{q}f(s),\tilde{\varphi}_\rho^{-1}\tilde{p}(s)\right)<C^\prime. 
\end{equation}
Let $C^{\prime\prime}>0$ be a constant satisfying $d\left(a(x,s_1),a(x,s_2)\right)\leq C^{\prime\prime}d(s_1,s_2)$ for all $s_1$, $s_2\in S$. Take any $b\in S/K_\rho$. Then we can show that 
\begin{equation*}
d_\mca{H}\left(a\left(x,\tilde{q}^{-1}(b)\right),\tilde{p}^{-1}\left(\tilde{\varphi}_\rho(b)\right)\right)\leq\max\left\{C,C^\prime C^{\prime\prime}\right\}
\end{equation*}
in the following way: 

\begin{claim}
$a\left(x,\tilde{q}^{-1}(b)\right)\subset B\left(\tilde{p}^{-1}\left(\tilde{\varphi}_\rho(b)\right),C\right)$. 
\end{claim}

For any $s\in\tilde{q}^{-1}(b)$, $d\left(\tilde{\varphi}_\rho(b),\tilde{p}\left(a(x,s)\right)\right)<C$ by \eqref{gosa}. Therefore $a(x,s)\in B\left(\tilde{p}^{-1}\left(\tilde{\varphi}_\rho(b)\right),C\right)$. 

\begin{claim}
$\tilde{p}^{-1}\left(\tilde{\varphi}_\rho(b)\right)\subset B\left(a\left(x,\tilde{q}^{-1}(b)\right),C^\prime C^{\prime\prime}\right)$. 
\end{claim}

For any $s\in\tilde{p}^{-1}\left(\tilde{\varphi}_\rho(b)\right)$, $d\left(\tilde{q}f(s),b\right)<C^\prime$ by \eqref{gosa2}, so that we can find $s^\prime\in\tilde{q}^{-1}(b)$ satisfying $d\left(s^\prime,f(s)\right)\leq C^\prime$. Then $d\left(a\left(x,s^\prime\right),s\right)\leq C^{\prime\prime}d\left(s^\prime,f(s)\right)\leq C^{\prime\prime}C^\prime$. 
\end{proof}

This proposition puts a strong restriction on the map $\tilde{\varphi}_\rho$ which we want to know, as in the next section or Section \ref{alsg}.

\subsection{Proof of Theorem \ref{ogs}}
As in Introduction, let 
\begin{equation*}
A_i=
\begin{pmatrix}
\alpha^{(i)}_1&&\\
&\ddots&\\
&&\alpha^{(i)}_n
\end{pmatrix}
\end{equation*}
for $i=1,\ldots,k$ be diagonal matrices with positive diagonal entries and set $S_A=\bb{R}^n\rtimes_{A(t)}\bb{R}^k$, where $A(t)=A_1^{t_1}\cdots A_k^{t_k}$ for $t=(t_1,\ldots,t_k)\in\bb{R}^k$. Ogasawara \cite{Og} proved the following theorem to classify up to quasiisometry groups obtained from $\bb{Z}^n$ by performing certain HNN extensions several times, which is a generalization of a result of Farb and Mosher \cite{FM} on quasiisometric rigidity of abelian-by-cyclic groups.  

\begin{thm}
If 
\begin{equation*}
\begin{tikzcd}
S_A\ar[r,"f"]\ar[d]&S_{A^\prime}\ar[d]\\
\bb{R}^k\ar[r,"\varphi"']&\bb{R}^k
\end{tikzcd}
\end{equation*}
is a diagram in which vertical maps are natural projections and $f$ is a fiber respecting quasiisometry over some linear map $\varphi$, then there exists a permutation matrix $P\in\GL(n,\bb{R})$ such that $PA(t)=A^\prime\left(\varphi(t)\right)P$ for every $t\in\bb{R}^k$. In particular, there exists a diagram 
\begin{equation*}
\begin{tikzcd}
S_A\ar[r,"\sim"]\ar[d]\ar[dr,phantom,"\circlearrowright" marking]&S_{A^\prime}\ar[d]\\
\bb{R}^k\ar[r,"\varphi"']&\bb{R}^k, 
\end{tikzcd}
\end{equation*}
where the above horizontal map is an isomorphism of Lie groups taking $(v,t)$ to $\left(Pv,\varphi(t)\right)$. 
\end{thm}

Consider an action $M\stackrel{\rho_0}{\curvearrowleft}S$ for $S=S_A$ and assume that for any $j$ there is some $i$ for which $\alpha^{(i)}_j\ne1$. It is easy to check that this is equivalent to $ [\mf{s},\mf{s}]=\bb{R}^n$. We take $\mf{h}=[\mf{s},\mf{s}]$. For any $\rho\in A(\mca{F},S)$, $\ker\varphi_\rho=[\mf{s},\mf{s}]$ and then $K_\rho=\bb{R}^n$. Therefore by Proposition \ref{key prop}, $a(x,\cdotp)$ is a fiber respecting biLipschitz diffeomorphism over $\tilde{\varphi}_\rho$ in 
\begin{equation*}
\begin{tikzcd}
S\ar[r,"{a(x,\cdotp)}"]\ar[d]&S\ar[d]\\
\bb{R}^k\ar[r,"\tilde{\varphi}_\rho"']&\bb{R}^k. 
\end{tikzcd}
\end{equation*}
Hence by the above theorem we get $\tilde{\varphi}_\rho\in W_S$. Using Theorem \ref{SC2}, this completes the proof of Theorem \ref{ogs}. 

\begin{rem}
There are solvable Lie groups for which this method does not work. Consider the solvable Lie group $S=\bb{R}^2\rtimes_{R(t)}\bb{R}$ for $R(t)=
\begin{pmatrix}
\cos2\pi t&-\sin2\pi t\\
\sin2\pi t&\cos2\pi t
\end{pmatrix}
$. Then the standard Euclidean metric on $\bb{R}^3$ identified with $S$ is a left invariant Riemannian metric on $S$. Hence for any $c\in\bb{R}\setminus\{0\}$ there is a biLipschitz diffeomorphism $f\colon S\to S$ such that 
\begin{equation*}
\begin{tikzcd}
S\ar[r,"f"]\ar[d]\ar[dr,phantom,"\circlearrowright" marking]&S\ar[d]\\
\bb{R}\ar[r,"\times c"']&\bb{R}. 
\end{tikzcd}
\end{equation*}
So we do not have any restriction on $c$. 
\end{rem}

\section{Parameter rigidity of suspension-like actions on mapping tori}
We prove Theorem \ref{mapping torus} here. The proof is based on Theorem \ref{SC2}, Proposition \ref{key prop} and a theorem of Farb and Mosher \cite{FM}. To calculate cohomologies we use Mayer--Vietoris exact sequences. 

The action which we consider is $M=\bb{Z}\backslash\left(\bb{T}^n\times\bb{R}\right)\curvearrowleft S$, where $S=V\rtimes_\Phi\bb{R}$. Let $\mca{F}$ be its orbit foliation.

\subsection{Cohomology}
The bracket operation of the Lie algebra $\mf{s}=V\rtimes_{\Phi_*}\bb{R}$ is 
\begin{align*}
\left[(v_1,t_1),(v_2,t_2)\right]&=\left([v_1,v_2]+\Phi_*(t_1)v_2-\Phi_*(t_2)v_1,[t_1,t_2]\right)\\
&=\left(\Phi_*(t_1)v_2-\Phi_*(t_2)v_1,0\right). 
\end{align*}
Let $T=(0,1)\in\mf{s}$. We will use the notation $\ad^0X=\ad X|_V$ for $X\in\mf{s}$. Then $\ad^0T=\Phi_*(1)$, so that $\Phi_t=e^{t\ad^0T}$ and $A|_V=e^{\ad^0T}$. Since we assume $1$ is not an eigenvalue of $A|_V$, $\ad^0T\colon V\to V$ does not have $0$ as an eigenvalue, hence it is invertible. So we have $[\mf{s},\mf{s}]=V$. 

Let us examine which cohomologies to compute. We take $[\mf{s},\mf{s}]=V$ as a subspace $\mf{h}$ appearing in Theorem \ref{SC}. We consider representations $\mf{s}\stackrel{\ad\circ\varphi}{\curvearrowright}\Gr_V(\mf{s})=\mf{s}/V\oplus V$ for $\varphi=\varphi_\rho$ for some $\rho\in A(\mca{F},S)$. The first component $\mf{s}\stackrel{\ad\circ\varphi}{\curvearrowright}\mf{s}/V$ is just the trivial representation. Let $\bar{T}$ be the image of $T$ under the natural projection $\mf{s}\twoheadrightarrow\mf{s}/V$ and put $\varphi(T)=c\bar{T}$. Then the second component $\mf{s}\stackrel{\ad\circ\varphi}{\curvearrowright}V$ is given by $(v,t)v^\prime=ct\left(\ad^0T\right)v^\prime=c\left(\ad^0(v,t)\right)v^\prime$, so that it is just $\mf{s}\stackrel{c\ad^0}{\curvearrowright}V$. Therefore we must compute 
\begin{equation*}
H^1(\mca{F})\quad\text{and}\quad H^1\left(\mca{F};\mf{s}\stackrel{c\ad^0}{\curvearrowright}V\right)
\end{equation*}
for all $c\in\bb{R}$. But we will see that vanishing of $H^1(\mca{F})$ and $H^1\left(\mca{F};\mf{s}\stackrel{\pm\ad^0}{\curvearrowright}V\right)$ suffices by the method of large scale geometry and this reduction fits our assumption of Zariski density nicely.

\subsection{Application of large scale geometry}\label{alsg}
Here we use rigidty of quasiisometries between solvable Lie groups found by Farb and Mosher. They prepare the following theorem as a tool to prove quasiisometic rigidity of abelian-by-cyclic groups. In \cite{FM} a slightly different statement (Theorem 5.2) is given, but the following is also proved. 

\begin{thm}\label{farb mosher}
Let $\Phi^{(1)}_t$ and $\Phi^{(2)}_t$ be one-parameter subgroups of $\GL(p,\bb{R})$ and $S_i=\bb{R}^p\rtimes_{\Phi^{(i)}}\bb{R}$.  If 
\begin{equation*}
\begin{tikzcd}
S_1\ar[r,"\phi"]\ar[d]&S_2\ar[d]\\
\bb{R}\ar[r,"\id"']&\bb{R}
\end{tikzcd}
\end{equation*}
is a diagram in which vertical maps are natural projections and $\phi$ is a fiber respecting quasiisometry over $\id$. Then absolute Jordan forms of $\Phi^{(1)}_1$ and $\Phi^{(2)}_1$ coincide up to permutation of Jordan blocks. Here an absolute Jordan form refers to a matrix obtained by replacing the diagonal entries of a Jordan normal form over the complex field with their absolute values. 
\end{thm}

Returning to our situation, we have 
\begin{equation*}
\begin{tikzcd}
S\ar[r,"{a(x,\cdotp)}"]\ar[d]&S\ar[d]\\
\bb{R}\ar[r,"\times c"']&\bb{R}, 
\end{tikzcd}
\end{equation*}
where $a(x,\cdotp)$ is a fiber respecting biLipschitz diffeomorphism by Proposition \ref{key prop}. Composing with 
\begin{equation*}
\begin{tikzcd}
S\ar[r,"\sim"]\ar[d]\ar[dr,phantom,"\circlearrowright" marking]&V\rtimes_{\Phi_{ct}}\bb{R}\ar[d]\\
\bb{R}\ar[r,"\times\frac{1}{c}"']&\bb{R}, 
\end{tikzcd}
\end{equation*}
where the above horizontal map takes $(v,t)$ to $\left(v,\frac{t}{c}\right)$, we get 
\begin{equation*}
\begin{tikzcd}
S\ar[r]\ar[d]&V\rtimes_{\Phi_{ct}}\bb{R}\ar[d]\\
\bb{R}\ar[r,"\id"']&\bb{R}, 
\end{tikzcd}
\end{equation*}
in which the above horizontal map is a fiber respecting biLipschitz diffeomorphism over $\id$. Then by Theorem \ref{farb mosher}, absolute Jordan forms of $\Phi_1$ and $\Phi_c$ coincide. Let $\alpha_1,\ldots,\alpha_p$ be the eigenvalues of $\ad^0T$, so that eigenvalues of $\Phi_t=e^{t\ad^0T}$ are $e^{t\alpha_1},\ldots,e^{t\alpha_p}$. Therefore the sets $\left\{\left|e^{\alpha_1}\right|,\ldots,\left|e^{\alpha_p}\right|\right\}$ and $\left\{\left|e^{\alpha_1}\right|^c,\ldots,\left|e^{\alpha_p}\right|^c\right\}$ coincide with multiplicity. Since $\Phi_1$ has an eigenvalue of absolute value not equal to $1$, we must have $c=\pm1$. So we need to verify only vanishing of $H^1(\mca{F})$ and $H^1\left(\mca{F};\mf{s}\stackrel{\pm\ad^0}{\curvearrowright}V\right)$ to show parameter rigidity of the action.

\subsection{Mayer--Vietoris argument}
Let $\mf{s}\stackrel{\pi}{\curvearrowright}W$ be a representation on a finite dimensional real vector space satisfying $\pi(v)=0$ for all $v\in V$. We will try to calculate the first cohomology $H^1\left(\mca{F};\mf{s}\stackrel{\pi}{\curvearrowright}W\right)$. 

Let $U_1$ and $U_2$ be the projections to $M$ of $\bb{T}^n\times(0,1)$ and $\bb{T}^n\times(-\frac{1}{2},\frac{1}{2})$, so that $M=U_1\cup U_2$. Then we have a short exact sequence of cochain complexes: 
\begin{equation*}
0\to\Omega^*(\mca{F};W)\to\Omega^*\left(\mca{F}|_{U_1};W\right)\oplus\Omega^*\left(\mca{F}|_{U_2};W\right)\to\Omega^*\left(\mca{F}|_{U_1\cap U_2};W\right)\to0. 
\end{equation*}
The second map maps $\xi$ to $\left(\xi|_{U_1},\xi|_{U_2}\right)$ and the third map maps $(\xi_1,\xi_2)$ to $\xi_2|_{U_1\cap U_2}-\xi_1|_{U_1\cap U_2}$. Hence we obtain a exact sequence of the cohomology
\begin{multline*}
H^0\left(\mca{F}|_{U_1};\pi\right)\oplus H^0\left(\mca{F}|_{U_2};\pi\right)\stackrel{P}{\to}H^0\left(\mca{F}|_{U_1\cap U_2};\pi\right)\\
\to H^1(\mca{F};\pi)\to H^1\left(\mca{F}|_{U_1};\pi\right)\oplus H^1\left(\mca{F}|_{U_2};\pi\right)\stackrel{Q}{\to} H^1\left(\mca{F}|_{U_1\cap U_2};\pi\right), 
\end{multline*}
so that we have 
\begin{equation}
H^1(\mca{F};\pi)\simeq\coker P\oplus\ker Q. \label{Hcokker}
\end{equation}

To compute $H^1(\mca{F};\pi)$, we first compute $H^1\left(\mca{F}|_{U_1};\pi\right)$. Let $\mca{G}$ be the orbit foliation of the translation action $\bb{T}^n\curvearrowleft V$. Define $\iota\colon\bb{T}^n\hookrightarrow U_1$ by $\iota(x)=\left(x,\frac{1}{2}\right)$. The map $\iota^*\colon\Omega^*\left(\mca{F}|_{U_1};W\right)\to\Omega^*(\mca{G};W)$ obtained by pull back is a cochain map\footnote{However, $p^*\colon\Omega^*(\mca{G};W)\to\Omega^*\left(\mca{F}|_{U_1};W\right)$ is not a cochain map, where $p\colon U_1\to\bb{T}^n$ maps $(x,t)$ to $x$. } between $\left(\Omega^*\left(\mca{F}|_{U_1};W\right),\df+\pi\omega_0\wedge\right)$ and $\left(\Omega^*(\mca{G};W),d_\mca{G}\right)$, where $\omega_0$ is the canonical 1--form of the action $M\curvearrowleft S$. In fact 
\begin{equation*}
\iota^*\left(\df+\pi\omega_0\wedge\right)\xi=d_\mca{G}\iota^*\xi+\iota^*\left(\pi\omega_0\wedge\xi\right)=d_\mca{G}\iota^*\xi
\end{equation*}
for $\xi\in\Omega^k\left(\mca{F}|_{U_1};W\right)$ due to our assumption on $\pi$. Thus we get the induced map on cohomologies: 
\begin{equation*}
\iota^*\colon H^*\left(\mca{F}|_{U_1};\pi\right)\to H^*(\mca{G};W). 
\end{equation*}
Note that $W$ in $H^*(\mca{G};W)$ is just a direct sum of the trivial coefficient $\bb{R}$. 

\begin{lem}\label{iso}
$\iota^*\colon H^1\left(\mca{F}|_{U_1};\pi\right)\to H^1(\mca{G};W)$ is an isomorphism. 
\end{lem}

\begin{proof}
We first show injectivity: 

\noindent
Take $[\eta]\in H^1\left(\mca{F}|_{U_1};\pi\right)$ satisfying $\iota^*[\eta]=0$. We must show the existence of a smooth $\alpha\colon U_1\to W$ such that 
\begin{equation}
\eta=\df\alpha+\pi\omega_0\alpha. \label{goal}
\end{equation}
Such $\alpha$ must satisfy $\eta(T)=T\alpha+\pi(T)\alpha$, where $T$ is regarded as a vector field using the action. Note that $T$ is just the vector field $\frac{\partial}{\partial t}$, where $(x,t)\in U_1$. We have a smooth map $h\colon\bb{T}^n\to W$ satisfying $\iota^*\eta=d_\mca{G}h$. Using this we let 
\begin{equation*}
\alpha(x,t)=e^{-\left(t-\frac{1}{2}\right)\pi(T)}\left(\int^t_\frac{1}{2}e^{\left(s-\frac{1}{2}\right)\pi(T)}\eta(T)(x,s)ds+h(x)\right)
\end{equation*}
for $x\in\bb{T}^n$ and $t\in(0,1)$. Then this $\alpha$ satisfies $\frac{\partial\alpha}{\partial t}(x,t)=-\pi(T)\alpha(x,t)+\eta(T)(x,t)$, ie $\eta(T)=T\alpha+\pi(T)\alpha$ and $\alpha\left(x,\frac{1}{2}\right)=h(x)$. To show (\ref{goal}) we have to prove 
\begin{equation*}
\eta(X)-X\alpha=0
\end{equation*}
for all $X\in V$. Note that 
\begin{equation}
\left(\eta(X)-X\alpha\right)\left(x,\frac{1}{2}\right)=0 \label{initial}
\end{equation}
for $X\in V$. Since 
\begin{equation*}
0=\left(\df\eta+\pi\omega_0\wedge\eta\right)(T,X)=T\eta(X)-X\eta(T)-\eta\left([T,X]\right)+\pi(T)\eta(X), 
\end{equation*}
we have 
\begin{align}
T\left(\eta(X)-X\alpha\right)&=X\eta(T)+\eta\left([T,X]\right)-\pi(T)\eta(X)-X\left(\eta(T)-\pi(T)\alpha\right)-[T,X]\alpha \notag \\
&=-\pi(T)\left(\eta(X)-X\alpha\right)+\eta\left([T,X]\right)-[T,X]\alpha. \label{eq}
\end{align}
Choose a basis of $V$ which turns $\ad^0T$ into a real Jordan normal form 
\begin{equation*}
\bigoplus
\begin{pmatrix}
a&1&&\\
&\ddots&\ddots&\\
&&\ddots&1\\
&&&a
\end{pmatrix}
\oplus\bigoplus
\begin{pmatrix}
b&-c&1&&&&&\\
c&b&&1&&&&\\
&&\ddots&&\ddots&&&\\
&&&\ddots&&\ddots&&\\
&&&&\ddots&&1&\\
&&&&&\ddots&&1\\
&&&&&&b&-c\\
&&&&&&c&b
\end{pmatrix}, 
\end{equation*}
where $a,b,c\in\bb{R}$, $c\neq 0$. 
Let $X_1,\ldots,X_k$ be the basis of the subspace corresponding to the above 
\begin{equation*}
\begin{pmatrix}
a&1&&\\
&\ddots&\ddots&\\
&&\ddots&1\\
&&&a
\end{pmatrix}, 
\end{equation*}
so that relations 
\begin{align*}
[T,X_1]&=aX_1\\
[T,X_2]&=aX_2+X_1\\
&\vdots\\
[T,X_k]&=aX_k+X_{k-1}
\end{align*}
hold. By (\ref{eq}), $T\left(\eta(X_1)-X_1\alpha\right)=\left(a-\pi(T)\right)\left(\eta(X_1)-X_1\alpha\right)$. The solution is 
\begin{equation*}
\left(\eta(X_1)-X_1\alpha\right)(x,t)=e^{\left(t-\frac{1}{2}\right)\left(a-\pi(T)\right)}\left(\eta(X_1)-X_1\alpha\right)\left(x,\frac{1}{2}\right). 
\end{equation*}
So we get $\eta(X_1)-X_1\alpha=0$ by (\ref{initial}). For $X_2$ we also have $T\left(\eta(X_2)-X_2\alpha\right)=\left(a-\pi(T)\right)\left(\eta(X_2)-X_2\alpha\right)$, so that $\eta(X_2)-X_2\alpha=0$. Repeating this we get $\eta(X_i)-X_i\alpha=0$ for $i=1,\ldots,k$. 

Let $X_1,Y_1,\ldots,X_k,Y_k$ be the basis of the subspace corresponding to 
\begin{equation*}
\begin{pmatrix}
b&-c&1&&&&&\\
c&b&&1&&&&\\
&&\ddots&&\ddots&&&\\
&&&\ddots&&\ddots&&\\
&&&&\ddots&&1&\\
&&&&&\ddots&&1\\
&&&&&&b&-c\\
&&&&&&c&b
\end{pmatrix}. 
\end{equation*}
We proceed similarly. The relations are 
\begin{align*}
[T,X_1]&=bX_1+cY_1,&[T,Y_1]&=-cX_1+bY_1,\\
[T,X_2]&=bX_2+cY_2+X_1,&[T,Y_2]&=-cX_2+bY_2+Y_1,\\
&\ \ \vdots&&\ \ \vdots
\end{align*}
The first equation to solve is 
\begin{equation*}
T
\begin{pmatrix}
\eta(X_1)-X_1\alpha\\
\eta(Y_1)-Y_1\alpha
\end{pmatrix}
=
\begin{pmatrix}
b-\pi(T)&c\\
-c&b-\pi(T)
\end{pmatrix}
\begin{pmatrix}
\eta(X_1)-X_1\alpha\\
\eta(Y_1)-Y_1\alpha
\end{pmatrix}, 
\end{equation*}
and the solution is 
\begin{equation*}
\begin{pmatrix}
\eta(X_1)-X_1\alpha\\
\eta(Y_1)-Y_1\alpha
\end{pmatrix}
(x,t)=e^{\left(t-\frac{1}{2}\right)
\begin{pmatrix}
b-\pi(T)&c\\
-c&b-\pi(T)
\end{pmatrix}
}
\begin{pmatrix}
\eta(X_1)-X_1\alpha\\
\eta(Y_1)-Y_1\alpha
\end{pmatrix}
\left(x,\frac{1}{2}\right). 
\end{equation*}
So we have $\eta(X_1)-X_1\alpha=0$ and $\eta(Y_1)-Y_1\alpha=0$. Repeating, we eventually get $\eta(X_i)-X_i\alpha=0$ and $\eta(Y_i)-Y_i\alpha=0$ for $i=1,\ldots,k$. This proves the injectivity. 

\vspace{2mm}
\noindent
We now prove surjectivity: 

\noindent
Take any $[\xi]\in H^1(\mca{G};W)$. We must construct $\eta\in\Omega^1\left(\mca{F}|_{U_1};W\right)$ satisfying 
\begin{equation}
\df\eta+\pi\omega_0\wedge\eta=0 \label{apple}
\end{equation}
and $\iota^*\eta=\xi$. We will construct $\eta$ requiring the additional property $\eta(T)=0$. In  order that equation (\ref{apple}) is satisfied, for any $X\in V$ we should construct $\eta(X)$ such that 
\begin{equation}
T\eta(X)-\eta\left([T,X]\right)+\pi(T)\eta(X)=0 \label{orange}
\end{equation}
holds. Fix a basis $X_1,\ldots,X_p$ of $V$ and let $(a_{ij})$ be the matrix representing $\ad^0T$ with respect to $X_1,\ldots,X_p$; $[T,X_j]=\sum^p_{i=1}a_{ij}X_i$. The $\eta(X_j)$ should satisfy $T\eta(X_j)-\sum^p_{i=1}a_{ij}\eta(X_i)+\pi(T)\eta(X_j)=0$ and $\eta(X_j)\left(x,\frac{1}{2}\right)=\xi\left(X^0_j\right)$. Here, for $X\in V$, $X^0$ denotes the section of $T\mca{G}$ satisfying $\iota_*X^0=X|_{t=\frac{1}{2}}$. So we must solve 
\begin{equation*}
T
\begin{pmatrix}
\eta(X_1)\\
\vdots\\
\eta(X_p)
\end{pmatrix}
=\left((a_{ji})-
\begin{pmatrix}
\pi(T)&&\\
&\ddots&\\
&&\pi(T)
\end{pmatrix}
\right)
\begin{pmatrix}
\eta(X_1)\\
\vdots\\
\eta(X_p)
\end{pmatrix}. 
\end{equation*}
The solution is 
\begin{equation*}
\begin{pmatrix}
\eta(X_1)\\
\vdots\\
\eta(X_p)
\end{pmatrix}
(x,t)=\exp\left\{\left(t-\frac{1}{2}\right)\left((a_{ji})-
\begin{pmatrix}
\pi(T)&&\\
&\ddots&\\
&&\pi(T)
\end{pmatrix}
\right)\right\}
\begin{pmatrix}
\xi\left(X^0_1\right)\\
\vdots\\
\xi\left(X^0_p\right)
\end{pmatrix}
(x). 
\end{equation*}
We define $\eta(X_i)$ by this. Then $\eta(X)$ for any $X\in V$ is defined by linearity, so that we have defined $\eta\in\Omega^1\left(\mca{F}|_{U_1};W\right)$. Then (\ref{orange}) and $\iota^*\eta=\xi$ are satisfied. To see that (\ref{apple}) holds, we only have to show $\left(\df\eta+\pi\omega_0\wedge\eta\right)(X,Y)=0$ for all $X, Y\in V$. Put $\theta=\df\eta+\pi\omega_0\wedge\eta$. Then
\begin{align*}
T\left(\theta(X,Y)\right)&=TX\eta(Y)-TY\eta(X)\\
&=X\left(\eta\left([T,Y]\right)-\pi(T)\eta(Y)\right)+[T,X]\eta(Y)\\
&\ \ \ \ \ \ \ -Y\left(\eta\left([T,X]\right)-\pi(T)\eta(X)\right)-[T,Y]\eta(X)\\
&=-\pi(T)\theta(X,Y)+\theta\left(X,[T,Y]\right)+\theta\left([T,X],Y\right)
\end{align*}
and $\theta(X,Y)\left(x,\frac{1}{2}\right)=X^0\xi\left(Y^0\right)-Y^0\xi\left(X^0\right)=d_\mca{G}\xi\left(X^0,Y^0\right)=0$. As in the proof of injectivity we can show $\theta(X_i,X_j)=0$, so that (\ref{apple}) is satisfied. This completes the proof. 
\end{proof}

By the assumption that $V$ is a Diophantine subspace of $\bb{R}^n$, the leafwise cohomology $H^1(\mca{G};W)$ vanishes, ie $H^1(\mca{G};W)=H^1(V;W)=V^*\otimes W$. Thus $H^1\left(\mca{F}|_{U_1};\pi\right)\simeq V^*\otimes W$. Let $q\colon\mf{s}\to V$ be the natural projection. 

\begin{lem}
An element $\varphi\otimes w\in V^*\otimes W$ corresponds to 
\begin{equation*}
\left(e^{\left(t-\frac{1}{2}\right)\left(\ad^0T\right)^*}\varphi\right)q\omega_0\otimes e^{-\left(t-\frac{1}{2}\right)\pi(T)}w\in\Omega^1\left(\mca{F}|_{U_1};W\right)
\end{equation*}
by the above isomorphism, where $\left(\ad^0T\right)^*\colon V^*\to V^*$ is the adjoint of $\ad^0T$. 
\end{lem}

\begin{proof}
Put $\eta=\left(e^{\left(t-\frac{1}{2}\right)\left(\ad^0T\right)^*}\varphi\right)q\omega_0\otimes e^{-\left(t-\frac{1}{2}\right)\pi(T)}w$. Then $\iota^*\eta=\varphi\otimes w$ and $\eta(T)=0$. For $X\in V$, $\eta(X)=\left(e^{\left(t-\frac{1}{2}\right)\left(\ad^0T\right)^*}\varphi\right)(X)e^{-\left(t-\frac{1}{2}\right)\pi(T)}w$ satisfies $T\eta(X)=\eta\left([T,X]\right)-\pi(T)\eta(X)$. 
So we see $\df\eta+\pi\omega_0\wedge\eta=0$ as in the proof of Lemma \ref{iso}.  
\end{proof}

$U_1\cap U_2$ is the disjoint union of the projections to $M$ of $\bb{T}^n\times\left(0,\frac{1}{2}\right)$ and $\bb{T}^n\times\left(\frac{1}{2},1\right)$. We define the maps 
\begin{align*}
\iota_\frac{1}{2}&\colon\bb{T}^n\hookrightarrow U_1,&\iota_\frac{1}{2}(x)&=\left(x,\frac{1}{2}\right),\\
\iota_0&\colon\bb{T}^n\hookrightarrow U_2,&\iota_0(x)&=(x,0),\\
\iota_\frac{1}{4}&\colon\bb{T}^n\hookrightarrow U_1\cap U_2,&\iota_\frac{1}{4}(x)&=\left(x,\frac{1}{4}\right),\\
\iota_\frac{3}{4}&\colon\bb{T}^n\hookrightarrow U_1\cap U_2,&\iota_\frac{3}{4}(x)&=\left(x,\frac{3}{4}\right). 
\end{align*}
We will calculate the bottom map of the next commutative diagram in which the vertical maps are isomorphisms: 
\begin{equation*}
\begin{tikzcd}
H^1\left(\mca{F}|_{U_1};\pi\right)\oplus H^1\left(\mca{F}|_{U_2};\pi\right)\ar[r,"Q"]\ar[d,"\iota_\frac{1}{2}^*\oplus\iota_0^*"']\ar[dr,phantom,"\circlearrowright" marking]&H^1\left(\mca{F}|_{U_1\cap U_2};\pi\right)\ar[d,"\iota_\frac{1}{4}^*\oplus\iota_\frac{3}{4}^*"]\\
(V^*\otimes W)\oplus(V^*\otimes W)\ar[r]&(V^*\otimes W)\oplus(V^*\otimes W). 
\end{tikzcd}
\end{equation*}

\begin{lem}
The bottom map of the above diagram is 
\begin{equation*}
\begin{pmatrix}
-e^{-\frac{1}{4}(\ad^0T)^*}\otimes e^{\frac{1}{4}\pi(T)}&e^{\frac{1}{4}(\ad^0T)^*}\otimes e^{-\frac{1}{4}\pi(T)}\\
-e^{\frac{1}{4}(\ad^0T)^*}\otimes e^{-\frac{1}{4}\pi(T)}&e^{-\frac{1}{4}(\ad^0T)^*}\otimes e^{\frac{1}{4}\pi(T)}
\end{pmatrix}. 
\end{equation*}
\end{lem}

\begin{proof}
We define $F\colon\bb{T}^n\times\bb{R}\to\bb{T}^n\times\bb{R}$ by $F(x,t)=\left(A^{-1}x,t-1\right)$. Take an element $\left(\varphi_1\otimes w_1,\varphi_2\otimes w_2\right)$ of $(V^*\otimes W)\oplus (V^*\otimes W)$. This is mapped by the vertical map to 
\begin{equation*}
\left(\left(e^{\left(t-\frac{1}{2}\right)\left(\ad^0T\right)^*}\varphi_1\right)q\omega_0\otimes e^{-\left(t-\frac{1}{2}\right)\pi(T)}w_1,\left(e^{t\left(\ad^0T\right)^*}\varphi_2\right)q\omega_0\otimes e^{-t\pi(T)}w_2\right), 
\end{equation*}
which is in turn mapped by $Q$ to 
\begin{multline*}
\left(\left(e^{t\left(\ad^0T\right)^*}\varphi_2\right)q\omega_0\otimes e^{-t\pi(T)}w_2-\left(e^{\left(t-\frac{1}{2}\right)\left(\ad^0T\right)^*}\varphi_1\right)q\omega_0\otimes e^{-\left(t-\frac{1}{2}\right)\pi(T)}w_1,\right.\\
\left.F^*\left(\left(e^{t\left(\ad^0T\right)^*}\varphi_2\right)q\omega_0\otimes e^{-t\pi(T)}w_2\right)-\left(e^{\left(t-\frac{1}{2}\right)\left(\ad^0T\right)^*}\varphi_1\right)q\omega_0\otimes e^{-\left(t-\frac{1}{2}\right)\pi(T)}w_1\right). 
\end{multline*}
Since 
\begin{equation*}
F^*\left(\left(e^{t\left(\ad^0T\right)^*}\varphi_2\right)q\omega_0\otimes e^{-t\pi(T)}w_2\right)=\left(e^{(t-1)\left(\ad^0T\right)^*}\varphi_2\right)q\omega_0\otimes e^{-(t-1)\pi(T)}w_2, 
\end{equation*}
the above element equals to 
\begin{multline*}
\left(\left(e^{\left(t-\frac{1}{4}\right)\left(\ad^0T\right)^*}e^{\frac{1}{4}\left(\ad^0T\right)^*}\varphi_2\right)q\omega_0\otimes e^{-\left(t-\frac{1}{4}\right)\pi(T)}e^{-\frac{1}{4}\pi(T)}w_2\right.\\
-\left(e^{\left(t-\frac{1}{4}\right)\left(\ad^0T\right)^*}e^{-\frac{1}{4}\left(\ad^0T\right)^*}\varphi_1\right)q\omega_0\otimes e^{-\left(t-\frac{1}{4}\right)\pi(T)}e^{\frac{1}{4}\pi(T)}w_1, \\
\left(e^{\left(t-\frac{3}{4}\right)\left(\ad^0T\right)^*}e^{-\frac{1}{4}\left(\ad^0T\right)^*}\varphi_2\right)q\omega_0\otimes e^{-\left(t-\frac{3}{4}\right)\pi(T)}e^{\frac{1}{4}\pi(T)}w_2\\
\left.-\left(e^{\left(t-\frac{3}{4}\right)\left(\ad^0T\right)^*}e^{\frac{1}{4}\left(\ad^0T\right)^*}\varphi_1\right)q\omega_0\otimes e^{-\left(t-\frac{3}{4}\right)\pi(T)}e^{-\frac{1}{4}\pi(T)}w_1
\right)
\end{multline*}
and finally this is mapped by the vertical arrow to 
\begin{multline*}
\left(e^{\frac{1}{4}\left(\ad^0T\right)^*}\varphi_2\otimes e^{-\frac{1}{4}\pi(T)}w_2-e^{-\frac{1}{4}\left(\ad^0T\right)^*}\varphi_1\otimes e^{\frac{1}{4}\pi(T)}w_1, \right.\\
\left.e^{-\frac{1}{4}\left(\ad^0T\right)^*}\varphi_2\otimes e^{\frac{1}{4}\pi(T)}w_2-e^{\frac{1}{4}\left(\ad^0T\right)^*}\varphi_1\otimes e^{-\frac{1}{4}\pi(T)}w_1\right). 
\end{multline*}
\end{proof}

By using the canonical isomorphism $V^*\otimes W\simeq\Hom(V,W)$, the map in the above lemma becomes a map 
\begin{equation*}
\bar{Q}\colon\Hom(V,W)\oplus\Hom(V,W)\to\Hom(V,W)\oplus\Hom(V,W)
\end{equation*}
which sends 
$\begin{pmatrix}
\alpha\\
\beta
\end{pmatrix}$
to 
\begin{equation*}
\begin{pmatrix}
-e^{\frac{1}{4}\pi(T)}\circ\alpha\circ e^{-\frac{1}{4}\ad^0T}+e^{-\frac{1}{4}\pi(T)}\circ\beta\circ e^{\frac{1}{4}\ad^0T}\\
-e^{-\frac{1}{4}\pi(T)}\circ\alpha\circ e^{\frac{1}{4}\ad^0T}+e^{\frac{1}{4}\pi(T)}\circ\beta\circ e^{-\frac{1}{4}\ad^0T}
\end{pmatrix}. 
\end{equation*}

\begin{lem}\label{kernel then}
If 
$\begin{pmatrix}
\alpha\\
\beta
\end{pmatrix}$ 
is in the kernel of $\bar{Q}$, then $\alpha$ and $\beta$ satisfy $e^{\pi(T)}\circ\alpha=\alpha\circ e^{\ad^0T}$ and $e^{\pi(T)}\circ\beta=\beta\circ e^{\ad^0T}$. 
\end{lem}

\begin{proof}
The pair 
$\begin{pmatrix}
\alpha\\
\beta
\end{pmatrix}$ 
is in the kernel if and only if $-e^{\frac{1}{2}\pi(T)}\circ\alpha+\beta\circ e^{\frac{1}{2}\ad^0T}=0$ and $-\alpha\circ e^{\frac{1}{2}\ad^0T}+e^{\frac{1}{2}\pi(T)}\circ\beta=0$. By eliminating $\beta$ or $\alpha$ we get the conclusion. 
\end{proof}

Next we will calculate the map $H^0\left(\mca{F}|_{U_1};\pi\right)\oplus H^0\left(\mca{F}|_{U_2};\pi\right)\stackrel{P}{\to} H^0\left(\mca{F}|_{U_1\cap U_2};\pi\right)$. The group $H^0\left(\mca{F}|_{U_1};\pi\right)$ consists of all smooth functions $f\colon U_1\to W$ satisfying $\df f+\pi\omega_0\wedge f=0$, which is equivalent to equations $Tf+\pi(T)f=0$ and $Xf=0$ for all $X\in V$. Since the action $\bb{T}^n\curvearrowleft V$ has a dense orbit, such $f$ must be constant along the directions of tori; write $f(x,t)=f(t)$. Solving the differential equation, we get $f(t)=e^{-\left(t-\frac{1}{2}\right)\pi(T)}f\left(\frac{1}{2}\right)$. Thus we have an isomorphism $H^0\left(\mca{F}|_{U_1};\pi\right)\simeq W$ which sends $f$ to $f\left(\frac{1}{2}\right)$. The bottom arrow of the diagram 
\begin{equation*}
\begin{tikzcd}
H^0\left(\mca{F}|_{U_1};\pi\right)\oplus H^0\left(\mca{F}|_{U_2};\pi\right)\ar[r,"P"]\ar[d,"\iota_\frac{1}{2}^*\oplus\iota_0^*"']\ar[dr,phantom,"\circlearrowright" marking]&H^0\left(\mca{F}|_{U_1\cap U_2};\pi\right)\ar[d,"\iota_\frac{1}{4}^*\oplus\iota_\frac{3}{4}^*"]\\
W\oplus W\ar[r,"\bar{P}"']&W\oplus W
\end{tikzcd}
\end{equation*}
sends an element $(w_1,w_2)$ as follows: 
\begin{align*}
(w_1,w_2)&\mapsto\left(e^{-\left(t-\frac{1}{2}\right)\pi(T)}w_1,e^{-t\pi(T)}w_2\right)\\
&\mapsto\left(e^{-t\pi(T)}w_2-e^{-\left(t-\frac{1}{2}\right)\pi(T)}w_1,F^*\left(e^{-t\pi(T)}w_2\right)-e^{-\left(t-\frac{1}{2}\right)\pi(T)}w_1\right)\\
&=\left(e^{-t\pi(T)}w_2-e^{-\left(t-\frac{1}{2}\right)\pi(T)}w_1,e^{-(t-1)\pi(T)}w_2-e^{-\left(t-\frac{1}{2}\right)\pi(T)}w_1\right)\\
&\mapsto\left(e^{-\frac{1}{4}\pi(T)}w_2-e^{\frac{1}{4}\pi(T)}w_1,e^{\frac{1}{4}\pi(T)}w_2-e^{-\frac{1}{4}\pi(T)}w_1\right). 
\end{align*}

An element $(w_1,w_2)$ is in $\ker\bar{P}$ if and only if $w_2=e^{\frac{1}{2}\pi(T)}w_1$ and $w_1=e^{\frac{1}{2}\pi(T)}w_2$. So 
\begin{equation}
\ker\bar{P}\simeq\ker\left(\id-e^{\pi(T)}\right)\subset W \label{Pker}
\end{equation}
by $(w_1,w_2)\mapsto w_2$. 

Finally we calculate the cohomology $H^1(\mf{s};\pi)$ of the Lie algebra. A 1--cocycle is a linear map $\varphi\colon\mf{s}\to W$ satisfying $\pi(T)\varphi(X)-\varphi\left([T,X]\right)=0$ for all $X\in V$.  So the space of 1--cocycles is isomorphic to $\Hom_T(V,W)\oplus W$ by the isomorphism which sends $\varphi$ to $\left(\varphi|_V,\varphi(T)\right)$, where $\Hom_T(V,W)$ denotes the space of $\left(\ad^0T,\pi(T)\right)$--equivariant linear maps from $V$ to $W$. On the other hand a 1--coboundary maps $Z\in\mf{s}$ to $\pi(Z)c\in W$ for some $c\in W$. So a linear map $\varphi\colon\mf{s}\to W$ is a 1--coboundary if and only if $\varphi(X)=0$ for all $X\in V$ and $\varphi(T)\in\im\pi(T)$.  The space of 1--coboundaries is isomorphic to $0\oplus\im\pi(T)$ by the above isomorphism. Therefore we have 
\begin{equation}
H^1(\mf{s};\pi)\simeq\Hom_T(V,W)\oplus\left(W/\im\pi(T)\right). \label{coh of Lie}
\end{equation}

\subsection{Vanishing of the cohomology}
As we saw in Section \ref{alsg} we must show vanishing of the following cohomologies: 
\begin{itemize}
\item[(a)] The trivial representation $\mf{s}\curvearrowright\bb{R}$. 
\item[(b)] The restriction of $\ad$ to $V$ and its negative, $\mf{s}\stackrel{\pm\ad^0}{\curvearrowright} V$. 
\end{itemize}

\noindent
(a)\ \ By (\ref{coh of Lie}) in last section, 
\begin{equation*}
H^1(\mf{s})\simeq\left\{\varphi\colon V\to\bb{R}\ \middle|\ \varphi\circ\ad^0T=0\right\}\oplus\bb{R}. 
\end{equation*}
By (\ref{Hcokker}) and (\ref{Pker}), $H^1(\mca{F})\simeq\bb{R}\oplus\ker\bar{Q}$. If $(\alpha,\beta)$ is in $\ker\bar{Q}$, by Lemma \ref{kernel then}, $\alpha=\alpha\circ\Phi_k$ for all $k\in\bb{Z}$. By the assumption of Zariski density, $\alpha=\alpha\circ\Phi_t$ for all $t\in\bb{R}$. Thus $\alpha\circ\ad^0T=0$ by differentiation. The same thing holds for $\beta$ and then we have $\alpha=\beta$. So 
\begin{equation*}
\ker\bar{Q}\simeq\left\{\varphi\colon V\to\bb{R}\ \middle|\ \varphi\circ\ad^0T=0\right\}
\end{equation*}
and $H^1(\mca{F})=H^1(\mf{s})$. 

\vspace{2mm}
\noindent 
b)\ \ By (\ref{coh of Lie}), 
\begin{equation*}
H^1\left(\mf{s};\pm\ad^0\right)\simeq\left\{\varphi\colon V\to V\ \middle|\ \varphi\circ\ad^0T=\pm\ad^0T\circ\varphi\right\}. 
\end{equation*}
By (\ref{Hcokker}) and (\ref{Pker}), $H^1\left(\mca{F};\pm\ad^0\right)\simeq\ker\bar{Q}$. If $(\alpha,\beta)$ is in $\ker\bar{Q}$, by Lemma \ref{kernel then}, $\Phi_{\pm k}\circ\alpha=\alpha\circ\Phi_k$ for all $k\in\bb{Z}$. By the assumption of Zariski density, $\Phi_{\pm t}\circ\alpha=\alpha\circ\Phi_t$ for all $t\in\bb{R}$. Thus $\pm\ad^0T\circ\alpha=\alpha\circ\ad^0T$ by differentiation. Note that here we use the benefit of large scale geometry. Since the map 
\begin{equation*}
\ker\bar{Q}\to\left\{\varphi\colon V\to V\ \middle|\ \varphi\circ\ad^0T=\pm\ad^0T\circ\varphi\right\}
\end{equation*}
mapping $(\alpha,\beta)$ to $\alpha$ is injective, we have $H^1\left(\mca{F};\pm\ad^0\right)=H^1\left(\mf{s};\pm\ad^0\right)$. 

This completes the proof of Theorem \ref{mapping torus}.

\section{Parameter rigidity of transitive locally free actions and rigidity of lattices}\label{rel}

\subsection{Relations between transitive locally free actions and lattices} 
Let $S$ be a connected simply connected solvable Lie group, $\Gamma_0$ be a cocompact lattice in $S$ and $M=\Gamma_0\backslash S\stackrel{\rho_0}{\curvearrowleft}S$ be the action by right multiplication. 

We put 
\begin{align*}
\mca{A}(\Gamma_0,S)&:=\left\{M\stackrel{\rho}{\curvearrowleft}S\ \middle|\ \text{$\rho$ is transitive locally free}\right\}\\
\mca{H}(\Gamma_0,S)&:=\left\{\alpha\colon\Gamma_0\to S\ \middle|\ \text{$\alpha$ is an injective homomorphism and $\alpha(\Gamma_0)$ is a lattice}\right\}. 
\end{align*}
Then 
\begin{equation*}
\mca{H}(\Gamma_0,S)/\Aut(\Gamma_0)=\left\{\Gamma\subset S\ \middle|\ \text{$\Gamma$ is a cocompact lattice isomorphic to $\Gamma_0$}\right\}. 
\end{equation*}

\begin{prop}
There is a one-to-one correspondence between 
\begin{equation*}
\mca{A}(\Gamma_0,S)/\left(\text{$C^\infty$--conjugacy $+$ $\Aut(S)$}\right)
\end{equation*}
and 
\begin{equation*}
\Aut(S)\backslash\mca{H}(\Gamma_0,S)/\Aut(\Gamma_0)
\end{equation*}
\begin{equation*}
=\Aut(S)\backslash\left\{\Gamma\subset S\ \middle|\ \text{$\Gamma$ is a cocompact lattice isomorphic to $\Gamma_0$}\right\}
\end{equation*}
taking $\rho$ to the isotropy subgroup of $\rho$ at the point $x_0=\Gamma_0$. 
\end{prop}

\begin{proof}
Well-definedness and injectivity are easy. For surjectivity, take any cocompact lattice $\Gamma$ in $S$ isomorphic to $\Gamma_0$. Then by a theorem of Mostow (See for example Raghunathan \cite[Theorem 3.6]{R}), $\Gamma\backslash S$ and $\Gamma_0\backslash S=M$ are diffeomorphic. So the map is also surjective. 
\end{proof}

\begin{prop}\label{corresp}
There is a one-to-one correspondence between 
\begin{equation*}
\mca{A}(\Gamma_0,S)/\text{parameter equivalence}
\end{equation*}
and 
\begin{equation*}
\Aut(S)\backslash\mca{H}(\Gamma_0,S). 
\end{equation*}
\end{prop}

\begin{proof}
The definition of the map is as follows: For a transitive locally free action $M\stackrel{\rho}{\curvearrowleft}S$, let $a_\rho\colon M\times S\to S$ be the smooth map defined uniquely by $\rho_0(x,s)=\rho\left(x,a_\rho(x,s)\right)$ and $a_\rho(x,1)=1$. Then $a_\rho$ is a cocycle over $\rho_0$: 
\begin{equation*}
a_\rho\left(x,ss^\prime\right)=a_\rho(x,s)a_\rho\left(\rho_0(x,s),s^\prime\right). 
\end{equation*}
Let $\Gamma_\rho$ be the isotropy subgroup of $\rho$ at $x_0$. Then we have $a_\rho(x_0,\cdotp)\colon\Gamma_0\stackrel{\sim}{\to}\Gamma_\rho\subset S$. So we will define the map by $\rho\mapsto a_\rho(x_0,\cdotp)$. 

Let us see the well-definedness of this map. Take two transitive locally free actions $M\stackrel{\rho_i}{\curvearrowleft}S$ $(i=1,2)$ which are parameter equivalent. So there are $\Phi\in\Aut(S)$ and a diffeomorphism $F\colon M\to M$ which is homotopic to the identity such that 
\begin{equation}\label{conjugate}
F\left(\rho_1(x,s)\right)=\rho_2\left(F(x),\Phi(s)\right). 
\end{equation}
Let $b\colon M\times S\to S$ be the cocycle over $\rho_1$ defined by $\rho_1(x,s)=\rho_0\left(x,b(x,s)\right)$. Note that $s=a_{\rho_1}\left(x,b(x,s)\right)$. We have $\rho_1(x,s)=\rho_2\left(x,a_{\rho_2}\left(x,b(x,s)\right)\right)$ and then $a_{\rho_2}\left(x,b(x,s)\right)$ is a cocycle over $\rho_1$. Since $F$ is homotopic to the identity we can define a smooth map $P\colon M\to S$ by $F(x)=\rho_2(x,P(x)^{-1})$. By \eqref{conjugate}, we see 
\begin{align*}
a_{\rho_2}\left(x,b(x,s)\right)&=P(x)^{-1}\Phi(s)P\left(\rho_1(x,s)\right)\\
&=P(x)^{-1}\Phi\left(a_{\rho_1}\left(x,b(x,s)\right)\right)P\left(\rho_0\left(x,b(x,s)\right)\right), 
\end{align*}
so that 
\begin{equation*}
a_{\rho_2}(x,s)=P(x)^{-1}\Phi\left(a_{\rho_1}(x,s)\right)P\left(\rho_0(x,s)\right)
\end{equation*}
for all $s\in S$ since $b(x,\cdotp)\colon S\to S$ is invertible. Taking $x=x_0$ and $s=\gamma\in\Gamma_0$, we have 
\begin{equation*}
a_{\rho_2}(x_0,\gamma)=P(x_0)^{-1}\Phi\left(a_{\rho_1}(x_0,\gamma)\right)P(x_0). 
\end{equation*}
So the map is well-defined. 

Next let us see the injectivity. Take two transitive locally free actions $M\stackrel{\rho_i}{\curvearrowleft}S$ for $i=1,2$ for which there exists $\Phi\in\Aut(S)$ such that $a_{\rho_2}(x_0,\gamma)=\Phi\left(a_{\rho_1}(x_0,\gamma)\right)$ for every $\gamma\in\Gamma_0$. Take $b\colon M\times S\to S$ satisfying $\rho_2(x,s)=\rho_0\left(x,b(x,s)\right)$, so that $s=b\left(x,a_{\rho_2}(x,s)\right)$. 
We have $S$--equivariant diffeomorphisms $\Gamma_{\rho_i}\backslash S\stackrel{\sim}{\to}M$ mapping $\Gamma_{\rho_i}s$ to $\rho_i(x_0,s)$. Using these define a diffeomorphism $F$ by 
\begin{equation*}
\begin{tikzcd}
M\ar[d,"F"']\ar[dr,phantom,"\circlearrowright" marking]&\Gamma_{\rho_1}\backslash S\ar[d,"\Phi"]\ar[l,"\sim"']\\
M&\Gamma_{\rho_2}\backslash S. \ar[l,"\sim"']
\end{tikzcd}
\end{equation*}
Here $\Phi$ denotes the map $\Gamma_{\rho_1}s\mapsto\Gamma_{\rho_2}\Phi(s)$. 
Obviously $F\left(\rho_1(x,s)\right)=\rho_2\left(F(x),\Phi(s)\right)$. To see that $F$ is homotopic to the identity, take any $\gamma\in\Gamma_0=\pi_1(M)$ and a curve $c\colon[0,1]\to S$ connecting $1$ and $\gamma$ and consider the curve $\rho_0\left(x_0,c(\cdotp)\right)\colon[0,1]\to M$. Since $\rho_0\left(x_0,c(t)\right)=\rho_1\left(x_0,a_{\rho_1}\left(x_0,c(t)\right)\right)$, we have 
\begin{align*}
F\left(\rho_0\left(x_0,c(t)\right)\right)&=\rho_2\left(x_0,\Phi\left(a_{\rho_1}\left(x_0,c(t)\right)\right)\right)\\
&=\rho_0\left(x_0,b\left(x_0,\Phi\left(a_{\rho_1}\left(x_0,c(t)\right)\right)\right)\right). 
\end{align*}
Therefore $F_*\colon\pi_1(M)\to\pi_1(M)$ maps $\gamma$ to 
\begin{equation*}
b\left(x_0,\Phi\left(a_{\rho_1}(x_0,\gamma)\right)\right)=b\left(x_0,a_{\rho_2}(x_0,\gamma)\right)=\gamma. 
\end{equation*}
So $F$ must be homotopic to the identity. 

Finally, we see the surjectivity. Take any $\alpha\colon\Gamma_0\to S$. Put $\Gamma:=\alpha(\Gamma_0)$ and $y_0:=\Gamma\in\Gamma\backslash S$. We have an isomorphism $\alpha\colon\Gamma_0\stackrel{\sim}{\to}\Gamma$ of the fundamental groups of $(M,x_0)$ and $(\Gamma\backslash S,y_0)$. By Witte \cite[Theorem 7.4]{W} which is a refinement of Mostow's theorem used in the proof of the previous proposition, we can find a diffeomorphism $F\colon(M,x_0)\to(\Gamma\backslash S,y_0)$ which induces $\alpha\colon\Gamma_0\stackrel{\sim}{\to}\Gamma$ on the fundamental groups. Define $M\stackrel{\rho}{\curvearrowleft}S$ by $\rho(x,s)=F^{-1}\left(F(x)s\right)$. Then $\rho$ is a transitive locally free action. We have $F\left(\rho_0(x_0,s)\right)=y_0a_\rho(x_0,s)$ for all $s\in S$. Take any $\gamma_0\in\Gamma_0$. Choose a curve $s\colon[0,1]\to S$ connecting $1$ and $\gamma_0$. Then the curve $\rho_0\left(x_0,s(t)\right)$ represents $\gamma_0$ in $\pi_1(M,x_0)$. Thus the curve $y_0a_\rho\left(x_0,s(t)\right)$ represents $F_*(\gamma_0)=\alpha(\gamma_0)$ in $\pi_1(\Gamma\backslash S,y_0)$. Since the curve $a_\rho\left(x_0,s(t)\right)$ in $S$ is the lift of the curve $y_0a_\rho\left(x_0,s(t)\right)$ starting from $1$, we get $a_\rho(x_0,\gamma_0)=\alpha(\gamma_0)$. This proves the surjectivity. 
\end{proof}

\begin{prop}
The map 
\begin{equation*}
\mca{A}(\Gamma_0,S)\to\mca{H}(\Gamma_0,S)
\end{equation*}
defined in Proposition \ref{corresp} is continuous. Here $\mca{A}(\Gamma_0,S)$ is endowed with the topology induced from the $C^\infty$ compact-open topology of $C^\infty(M\times S,M)$ and $\mca{H}(\Gamma_0,S)$ has the topology of pointwise convergence. 
\end{prop}

\begin{proof}
Take any $\gamma_0\in\Gamma_0$. We must show that $a_\rho(x_0,\gamma_0)\in S$ is continuous with respect to $\rho$. The map $\mca{A}(\Gamma_0,S)\to C^\infty(S,M)$ which sends $\rho$ to $\rho(x_0,\cdot)$ is continuous. Take a $C^\infty$--curve $c\colon[0,1]\to S$ connecting $1$ and $\gamma_0$. Let $c_\rho\colon[0,1]\to S$ be the lift of the curve $\rho_0\left(x_0,c(t)\right)$ with respect to the covering $\rho(x_0,\cdot)\colon S\to M$ starting at $1$. Then $a_\rho(x_0,\gamma_0)=c_\rho(1)$. Then the map $\mca{A}(\Gamma_0,S)\to C^\infty\left([0,1],S\right)$ taking $\rho$ to $c_\rho$ is continuous. In particular, $\rho\mapsto c_\rho(1)=a_\rho(x_0,\gamma_0)$ is continuous. 
\end{proof}

\begin{prop}
We have a commutative diagram 
\begin{equation*}
\begin{tikzcd}
\mca{A}(\Gamma_0,S)/\text{parameter equivalence}\ar[d,twoheadrightarrow]\ar[dr,phantom,"\circlearrowright" marking]\ar[r,"\sim"]&\Aut(S)\backslash\mca{H}(\Gamma_0,S)\ar[d,twoheadrightarrow]\\
\mca{A}(\Gamma_0,S)/\left(\text{$C^\infty$--conjugacy $+$ $\Aut(S)$}\right)\ar[r,"\sim"]&\Aut(S)\backslash\mca{H}(\Gamma_0,S)/\Aut(\Gamma_0), 
\end{tikzcd}
\end{equation*}
where the horizontal maps are defined above and the vertical surjective maps are defined obviously. 
\end{prop}

\begin{proof}
This is obvious. 
\end{proof}

By combining Proposition \ref{corresp} with Theorem \ref{SC}, we get the following. 

\begin{cor}
Let $S$ be a connected simply connected solvable Lie group, $\Gamma_0$ be a lattice in $S$ and $\mf{h}$ be a subspace between $[\mf{s},\mf{s}]$ and $\mf{n}$. If 
\begin{equation*}
H^1\left(\Gamma_0\backslash S;\mf{s}\stackrel{\ad\circ\varphi}{\curvearrowright}\Gr_\mf{h}(\mf{s})\right)=H^1\left(\mf{s};\mf{s}\stackrel{\ad\circ\varphi}{\curvearrowright}\Gr_\mf{h}(\mf{s})\right)
\end{equation*}
for all surjective Lie algebra homomorphisms $\varphi\colon\mf{s}\to\mf{s}/\mf{h}$, then any injective homomorphism $\alpha\colon\Gamma_0\to S$ whose image is a lattice, is transformed into the inclusion $\Gamma_0\hookrightarrow S$ by an element of $\Aut(S)$. In particular, if $\Gamma$ is a lattice in $S$ isomorphic to $\Gamma_0$, there is an isomorphism of $S$ which transforms $\Gamma$ into $\Gamma_0$. 
\end{cor}

\subsection{A counterexample to Theorem \ref{nil} for solvable Lie groups}
Now we see a counterexample to Theorem \ref{nil} for solvable Lie groups. The next example is taken from Example 2.2 of Baues and Klopsch \cite{BK}, which is due to Milovanov \cite{Milo}. Consider $A_0=
\begin{pmatrix}
0&-1\\
1&3
\end{pmatrix}
\in\SL(2,\bb{Z})$ which has $\lambda=\frac{3+\sqrt{5}}{2}$ and $\lambda^{-1}=\frac{3-\sqrt{5}}{2}$ as eigenvalues and let 
\begin{equation*}
X(t)=
\begin{pmatrix}
\lambda^t\cos2\pi t&-\lambda^t\sin2\pi t&&\\
\lambda^t\sin2\pi t&\lambda^t\cos2\pi t&&\\
&&\lambda^{-t}&\\
&&&\lambda^{-t}
\end{pmatrix}
\end{equation*}
for $t\in\bb{R}$. Since $X(1)$ and 
$A=
\begin{pmatrix}
A_0&\\
&A_0
\end{pmatrix}
\in\SL(4,\bb{Z})$ 
are conjugate, the solvable Lie group $S=\bb{R}^4\rtimes_{X(t)}\bb{R}$ contains a lattice $\Gamma_0$ isomorphic to $\bb{Z}^4\rtimes_{A}\bb{Z}$. Then the set 
\begin{equation*}
\mca{H}(\Gamma_0,S)/\Aut(S)
\end{equation*}
is uncountable as in Baues and Klopsch \cite{BK}. Therefore by Proposition \ref{corresp} the set 
\begin{equation*}
\mca{A}(\Gamma_0,S)/\text{parameter equivalence}
\end{equation*}
is uncountable, so that the natural action $\Gamma_0\backslash S\curvearrowleft S$ is not parameter rigid. However by computing $\Gamma_0/[\Gamma_0,\Gamma_0]$ using $\Gamma_0\simeq\bb{Z}^4\rtimes_{A}\bb{Z}$ we can show $H^1(\Gamma_0\backslash S)=H^1(\mf{s})$.

\subsection{A locally parameter rigid action of a contractible group which is not parameter rigid}
Let $M\stackrel{\rho_0}{\curvearrowleft}S$ be an action with the orbit foliation $\mca{F}$. The set of all smooth actions $M\curvearrowleft S$ with the orbit foliation $\mca{F}$ is denoted by $A(\mca{F},S)$, which is endowed with the topology induced from the $C^\infty$ compact-open topology of $C^\infty(M\times S,M)$. An action $M\stackrel{\rho_0}{\curvearrowleft}S$ is said to be {\em locally parameter rigid} if any $\rho\in A(\mca{F},S)$ which is close enough to $\rho_0$ is parameter equivalent to $\rho_0$. In \cite[Section 1.1.2]{A}, Asaoka comments that there is no known locally parameter rigid action of a contractible group which is not parameter rigid. Here we give an example of such an action. In \cite[Example 2.5]{BK}, Baues and Klopsch give an example of a connected simply connected solvable Lie group $S$ which has a Zariski dense lattice $\Gamma_0$ such that $\Aut(S)\backslash\mca{H}(\Gamma_0,S)$ is countably infinite. Here Zariski density means that the Zariski closures of $\Ad(\Gamma_0)$ and $\Ad(S)$ in $\GL(\mf{s})$ coincide. If we take such $S$ and $\Gamma_0$ then the action $\Gamma_0\backslash S\stackrel{\rho_0}{\curvearrowleft}S$ by right multiplication is not parameter rigid by Proposition \ref{corresp}. Let us see that this action is locally parameter rigid. Since $\Gamma_0$ is a Zariski dense lattice in $S$, the inclusion $\iota\colon\Gamma_0\hookrightarrow S$ is {\em locally rigid}, that is, the $\Aut(S)$--orbit of $\iota$ in $\mca{H}(\Gamma_0,S)$ is a neighborhood of $\iota$. See for example Baues and Klopsch \cite[Theorem 1.9]{BK}. Take an open neighborhood $U$ of $\iota$ in $\mca{H}(\Gamma_0,S)$ which is contained in the $\Aut(S)$--orbit of $\iota$. Let $V$ be the inverse image of $U$ by the map $\mca{A}(\Gamma_0,S)\to\mca{H}(\Gamma_0,S)$, then $V$ is an open neighborhood of $\rho_0$ in $\mca{A}(\Gamma_0,S)$ by continuity. Since $U$ projects to a one-point set in $\Aut(S)\backslash\mca{H}(\Gamma_0,S)$, $V$ also projects to a one-point set in $\mca{A}(\Gamma_0,S)/\text{(parameter equivalence)}$. Therefore $\rho_0$ is locally parameter rigid.


\begin{thebibliography}{99}
\bibitem{A}M. Asaoka. 
{\em Deformation of locally free actions and leafwise cohomology}. 
Foliations: Dynamics, Geometry and Topology, 1--40, 
Adv. Courses Math. CRM Barcelona, 
Birkh\"auser/Springer, Basel, 2014. 
\bibitem{BK}O. Baues and B. Klopsch. 
{\em Deformations and rigidity of lattices in solvable Lie groups}. 
J. Topol. 
\textbf{6}(2013), 823--856. 
\bibitem{FM}B. Farb and L. Mosher. 
{\em On the asymptotic geometry of abelian-by-cyclic groups}. 
Acta Math. 
\textbf{184}(2000), 145--202. 
\bibitem{Ma}H. Maruhashi. 
{\em Parameter rigid actions of simply connected nilpotent Lie groups}. 
Ergod. Th. \& Dynam. Sys. 
\textbf{33}(2013), 1864--1875. 
\bibitem{Ma3}H. Maruhashi. 
{\em Vanishing of cohomology and parameter rigidity of actions of solvable Lie groups, II}. 
Ergod. Th. \& Dynam. Sys. 
(2020), 1--37. doi:10.1017/etds.2020.97
\bibitem{MM}S. Matsumoto and Y. Mitsumatsu. 
{\em Leafwise cohomology and rigidity of certain Lie group actions}. 
Ergod. Th. \& Dynam. Sys. 
\textbf{23}(2003), 1839--1866. 
\bibitem{Milo}M. V. Milovanov. 
{\em The extension of automorphisms of uniform discrete subgroups of solvable Lie groups}. (Russian)
Dokl. Akad. Nauk BSSR. 
\textbf{17}(1973), 892--895, 969. 
\bibitem{Og}N. Ogasawara. 
{\em Quasiisometric classification of groups obtained from $\bb{Z}^n$ by HNN extension performed several times}. (Japanese)
Master Thesis, Kyoto University, 2012. 
\bibitem{R}M. S. Raghunathan. 
{\em Discrete subgroups of Lie groups}. 
Springer-Verlag, 1972. 
\bibitem{W}D. Witte. 
{\em Superrigidity of lattices in solvable Lie groups}. 
Invent. math. 
\textbf{122}(1995), 147--193. 
\end{thebibliography}
\end{document}